\documentclass[a4paper,reqno, 10pt]{amsart}

\usepackage{amsmath,amssymb,amsfonts,amsthm, mathrsfs}

\usepackage{verbatim,wasysym,cite}

\usepackage[latin1]{inputenc}
\usepackage{microtype}
\usepackage{color,enumitem,graphicx}
\usepackage[colorlinks=true,urlcolor=blue, citecolor=red,linkcolor=blue,
linktocpage,pdfpagelabels, bookmarksnumbered,bookmarksopen]{hyperref}
\usepackage[english]{babel}
\usepackage[symbol]{footmisc}
\renewcommand{\epsilon}{{\varepsilon}}

\numberwithin{equation}{section}
\newtheorem{theorem}{Theorem}[section]
\newtheorem{lemma}[theorem]{Lemma}
\newtheorem{remark}[theorem]{Remark}

\newtheorem{proposition}[theorem]{Proposition}
\newtheorem{corollary}[theorem]{Corollary}

\newcommand{\C}{\mathbb C}

\newcommand{\R}{\mathbb R}
\newcommand{\N}{\mathbb N}

\def\({\left(}
\def\){\right)}
\def\<{\left\langle}
\def\>{\right\rangle}

\def\Sch{{\mathcal S}}
\def\Pch{{\mathcal P}}

\DeclareMathOperator{\RE}{Re}
\DeclareMathOperator{\IM}{Im}

\begin{document}

\title[Scattering for the 1D NLS with point nonlinearity]{Blow-up and scattering for the 1D NLS with point nonlinearity above the mass-energy threshold}

\author[Alex H. Ardila]{Alex H. Ardila}
\address{Universidade Federal de Minas Gerais\\ ICEx-UFMG\\ CEP
30123-970\\ MG, Brazil} 
\email{ardila@impa.br}

\begin{abstract}
In this paper, we study the nonlinear Schr\"odinger equation with focusing point nonlinearity
in dimension one. First, we establish a scattering criterion for the equation based on Kenig-Merle's compactness-rigidity argument.
Then we prove the energy scattering below and above the mass-energy threshold.  We also describe the  dynamics of solutions
with data at the ground state threshold. Finally, we prove a blow-up criteria for the equation with initial data with  
arbitrarily large energy.
\end{abstract}


\maketitle

\medskip

\section{Introduction}
\label{sec:intro}
The nonlinear Schr\"odinger equations with point interactions have been intensively studied in recent years \cite{LFF, RJJ, FO, GHW, IkeaInu2017}. 
This interest is motivated by physical experiments in the theory of Bose-Einstein condensates and in nonlinear optics; see \cite{BMQT, QGAH, PDEOM, MakMaloCho2003, AFI} and references therein. In this paper, we consider the Cauchy problem for the following nonlinear Schr\"odinger equation with focusing point nonlineariy
\begin{equation}\label{NLS}
\begin{cases} 
i\partial_{t}u+\partial^{2}_{x}u+\delta(x)|u|^{p-1}u=0, \quad \text{$x\in \R$, $t\in \R$, } \\
u(0)=u_{0}\in H^{1}(\R),
\end{cases} 
\end{equation}
where $u:\mathbb{R}\times\mathbb{R}\rightarrow \mathbb{C}$, $\delta(x)$ is the Dirac mass at $x=0$ and the nonlinearity is $L^{2}$-supercritical, i.e. $p>3$. The dirac measure is used to model a defect localized at the origin (see, for example, \cite{MakMaloCho2003}).

The unique local existence of solutions is well known (see \cite[Theorem 1.1]{HOLMER2020123522}): Given $u_{0}\in H^{1}(\R)$, there exists a unique solution $u\in C([0,T_{+}), H^{1}(\R))$ of the Cauchy problem \eqref{NLS} for some maximal existence 
interval $[0, T_{+})$. We say that the solution $u$ is global forward in time if $T_{+}=\infty$. Besides, the mass $M(u(t))$ and energy $E(u(t))$ are independent of $t$, where
\begin{align*}
M(f)=\int_{\R}| f(x)|^{2}dx, \quad E(f)=\frac{1}{2}K(f)-\frac{1}{p+1}N(f),
\end{align*}
with
\[
K(f)=\int_{\R}|\partial_{x} f(x)|^{2}dx\quad
\text{and}
\quad
N(f)=|f(0)|^{p+1}
\]
for  $f\in H^{1}(\R)$. The equation \eqref{NLS} has a scaling invariance: if $u(x,t)$ is a solution of the Cauchy problem \eqref{NLS}, then 
\begin{equation}\label{Scale-inv}
\lambda^{\frac{1}{p-1}}u(\lambda x, \lambda^{2}t), \quad \lambda>0
\end{equation}
is also a solution with the rescaled initial data. Thus, by scaling argument it is possible to show that $\dot{H}^{\gamma_{c}}$ is the critical 
Sobolev space to \eqref{NLS}, with
\[
\gamma_{c}=\frac{1}{2}-\frac{1}{p-1}.
\]
Therefore, the NLS \eqref{NLS} is called $L^{2}$-subcritical if $p<3$, $L^{2}$-critical if $p=3$ and $L^{2}$-supercritical if $p>3$.
We say that the solution $u\in C([0,\infty), H^{1}(\R))$ to \eqref{NLS} scatters in $H^{1}(\R)$ forward in time, if there exists $\psi^{+}\in H^{1}(\R)$
such that
\[
\lim_{t\to\infty}\|u(t)-e^{it \partial^{2}_{x}}\psi^{+}\|_{H^{1}(\R)}=0.
\]
The equation \eqref{NLS} admits a global but nonscattering solution $u(x,t)=e^{it}Q(x)$, where $Q\neq 0$ should satisfy the following elliptic equation
\begin{equation}\label{EllipticE}
-\partial^{2}_{x}Q+Q-\delta(x)|Q|^{p-1}Q=0, \quad x\in \R.
\end{equation}
In \cite{HOLMER2020123522}, the authors prove that there exists a unique positive symmetric solution of the stationary problem \eqref{EllipticE}, which is given by
\begin{equation}\label{Groundsta}
Q(x)=2^{\frac{1}{p-1}}e^{-|x|}.
\end{equation}
We call $Q$ the ground state.  Recently,  Adami-Fukuizumi-Holmer\cite{AdamiFukuHolmer} determined the long time dynamics to \eqref{NLS} with data \textit{below} the ground state threshold.

 Our purpose in this paper is to discuss the the global behavior of the solutions to \eqref{NLS} \textit{at} and \textit{above} the mass and energy ground states threshold. With this in mind, we give the following unified scattering criterion for the equation \eqref{NLS}.
Using this criterion, we will be able to obtain scattering below, at and above the mass and energy ground states threshold. 
We set
\[
\sigma_{c}:=\frac{1-\gamma_{c}}{\gamma_{c}}.
\]
\begin{theorem}[Scattering criterion]\label{Th1}
Let $p>3$ and  $u_{0}\in H^{1}(\R)$. Let $u(t)$ the corresponding solution of the Cauchy problem 
\eqref{NLS} with initial data $u_{0}$ defined on the maximal forward time lifespan $[0, T_{+})$. Assume that
\begin{equation}\label{InequSc}
\sup_{t\in [0, T_{+})}N(u(t))[M(u(t))]^{\sigma_{c}}< N(Q)[M(Q)]^{\sigma_{c}},
\end{equation}
then the solution $u(t)$ is global ($T_{+}=\infty$) and scatters in $H^{1}(\R)$ forward in time. 
\end{theorem}
For the classical NLS, a similar result was originally proven by Duyckaerts-Roudenko \cite[Theorem 3.7]{DR5} through the use concentration-compactness-rigidity argument of Kenig-Merle \cite{KenigMerle2006}. Later, in Dinh \cite{Dinh2020}, this result in  dimension $N\geq 3$ was proven using the Dodson-Murphy' argument \cite{DodsonMurphy2018}, which simplifies the process of the proof for the scattering. For other results in this direction see also \cite{Tarek2020, GaoWang2020}. Naturally, in this paper, we establish the scattering criterion (Theorem \ref{Th1}) for the equation \eqref{NLS} based on the argument of Duyckaerts-Roudenko \cite{DR5}.

As a first consequence of the Theorem \ref{Th1}, we have the energy scattering below the ground state threshold, which 
was originally proved by  Adami-Fukuizumi-Holmer\cite{AdamiFukuHolmer}  based on the ideas of \cite{HolmerRoudenko2008, DuyHolmerRoude2008}.

\begin{theorem}[Scattering below the threshold, \cite{AdamiFukuHolmer}]\label{ScattebelowQ}
Let $p>3$ and  $u_{0}\in H^{1}(\R)$ satisfy
\begin{equation}\label{EMass}
E(u_{0})[M(u_{0})]^{\sigma_{c}}<E(Q)[M(Q)]^{\sigma_{c}}
\end{equation}
and
\begin{equation}\label{KinectMass}
K(u_{0})[M(u_{0})]^{\sigma_{c}}<K(Q)[M(Q)]^{\sigma_{c}}.
\end{equation}
Then the solution $u(t)$ of Cauchy problem \eqref{NLS} is global and scatters in $H^{1}(\R)$ in both directions.
\end{theorem}

\begin{remark}\label{Blow-up}
We observe that if $u_{0}\in H^{1}(\R)$ satisfies \eqref{EMass} and $K(u_{0})[M(u_{0})]^{\sigma_{c}}>K(Q)[M(Q)]^{\sigma_{c}}$, then
the  corresponding solution  $u(t)$ of the Cauchy problem \eqref{NLS} blows-up  in  the  positive  time  direction. 
An analogous statement holds for negative time; see \cite[Theorem 1.4]{HOLMER2020123522} for more details.
\end{remark}

In the following result, we investigate the long time dynamics for the equation \eqref{NLS}
at mass and energy ground states threshold $E(u_{0})[M(u_{0})]^{\sigma_{c}}=E(Q)[M(Q)]^{\sigma_{c}}$.
Indeed, using the scattering criterion Theorem \ref{Th1} and the compactness of minimizing sequence for the Gagliardo-Nirenberg inequality
(see \eqref{Ga-Ni-Inequ} below) we obtain the following result.

\begin{theorem}[Dynamics at  threshold]\label{DynaThre}
Let $p>3$ and  $u_{0}\in H^{1}(\R)$ satisfy
\begin{equation}\label{EMassThre}
E(u_{0})[M(u_{0})]^{\sigma_{c}}=E(Q)[M(Q)]^{\sigma_{c}}.
\end{equation}
(i) If 
\begin{equation}\label{DynaKinectMass}
K(u_{0})[M(u_{0})]^{\sigma_{c}}<K(Q)[M(Q)]^{\sigma_{c}},
\end{equation}
then one of the following cases holds. 
\begin{itemize}[leftmargin=5mm]
		\item The corresponding solution $u(t)$ to \eqref{NLS} scatters in $H^{1}(\R)$ forward in time.
		\item  There exist $\theta\in \R$ and a time sequence $t_{n}\rightarrow \infty$ such that
		\begin{equation}\label{CoverGround}
    u(\cdot, t_{n})\rightarrow 2^{\frac{1}{p-1}}e^{i\theta}e^{-|\cdot|}  \quad \text{in $H^{1}(\R)$ as $n\rightarrow\infty$.}
     \end{equation}
		\end{itemize}
(ii) If
\begin{equation}\label{IgualKinectMass}
K(u_{0})[M(u_{0})]^{\sigma_{c}}=K(Q)[M(Q)]^{\sigma_{c}},
\end{equation}
then there exists $\theta\in \R$ such that the corresponding solution  $u(t)$ to \eqref{NLS} is given by 
$u(x,t)=2^{\frac{1}{p-1}}e^{i\theta}e^{-|x|}$.\\
(iii) If 
\begin{equation}\label{MayorlKinectMass}
K(u_{0})[M(u_{0})]^{\sigma_{c}}>K(Q)[M(Q)]^{\sigma_{c}},
\end{equation}
then one of the following cases holds. 
\begin{itemize}[leftmargin=5mm]
		\item The corresponding solution $u(t)$ to \eqref{NLS} blows-up in finite time.
		\item  The solution $u(t)$ is global and exist $\rho\in \R$ and a time sequence $\tau_{n}\rightarrow \infty$ such that
		\begin{equation}\label{CoverGroundnew}
    u(\cdot, \tau_{n})\rightarrow 2^{\frac{1}{p-1}}e^{i\rho}e^{-|\cdot|} \quad \text{in $H^{1}(\R)$ as $n\rightarrow\infty$.}
     \end{equation}
	\end{itemize}
\end{theorem}

Next, we have interested in a criteria of scattering that includes initial data above the mass-energy threshold. We recall that when 
$u_{0}\in \Sigma:=\left\{f\in H^{1}(\R): |x|f\in L^{2}(\R)\right\}$, the virial quantity of \eqref{NLS},
\[
	V(t):= \int_{\R} x^2 |u(x,t)|^2  dx,
\]
is finite for all $t\in [0, T_{+})$ and satisfies the virial identity
\[
V^{\prime\prime}(t)=G(u) \quad \text{where}\quad G(u):=8K(u)-4N(u)
\]
is the Pohozaev functional. 

As another consequence of Theorem \ref{Th1}, we obtain the following scattering result for \eqref{NLS}
above the mass-energy threshold.  An analogous result was proven in \cite{DR5} for the classic NLS.  

\begin{theorem}[Scattering above the threshold]\label{ScatteaboveQ}
Let $p>3$ and  $u_{0}\in \Sigma(\R)$.  Assume that 
\begin{align}\label{assum11}
&	E(u_{0})[M(u_{0})]^{\sigma_{c}}\geq E(Q)[M(Q)]^{\sigma_{c}},\\\label{assump22}
&\frac{E(u_{0})[M(u_{0})]^{\sigma_{c}}}{E(Q)[M(Q)]^{\sigma_{c}}}\(1-\frac{(V^{\prime}(0))^{2}}{32E(u_{0})V(0)}\)\leq 1,\\
\label{assum33}
& N(u_{0})[M(u_{0})]^{\sigma_{c}}<N(Q)[M(Q)]^{\sigma_{c}},\\\label{assump44}
&V^{\prime}(0)\geq 0.
\end{align}
Then the solution $u(t)$ with initial data $u_{0}$ is global and scatters forward in time in $H^{1}(\R)$.
\end{theorem}

Complementing the theorem stated above, we extend the scope of blow-up solutions in Remark \ref{Blow-up} to those with arbitrarily large energy.
Indeed, we have the following finite time blow-up result for the equation \eqref{NLS}.

\begin{theorem}[Blow up above the threshold]\label{BlowupboveQ}
Let $u_{0}\in \Sigma(\R)$ satisfy \eqref{assump22},
\begin{align}
\label{Blowup22}
& N(u_{0})[M(u_{0})]^{\sigma_{c}}>N(Q)[M(Q)]^{\sigma_{c}},\\\label{Blowup33}
&V^{\prime}(0)\leq  0.
\end{align}
Then the solution $u(t)$ to Cauchy problem \eqref{NLS} with initial data $u_{0}$ blows-up forward in finite time.
\end{theorem}

As a consequence of the Theorems \ref{ScatteaboveQ} and \ref{BlowupboveQ}, we classify the behavior of the ground state  modulated by a
quadratic phase. More specifically,

\begin{corollary}\label{GroundPhase}
Let $\gamma\in \R$ and let $\psi^{\gamma}(t)$ be the corresponding solution to \eqref{NLS} with initial data 
\[
\psi^{\gamma}_{0}(x)=2^{\frac{1}{p-1}}e^{i\gamma x^{2}}e^{-|x|}.
\]
Then, for $\gamma>0$ the solution $\psi^{\gamma}(t)$ is globally defined on $[0, \infty)$, scatters forward in $H^{1}(\R)$ and blows up in negative time. Furthermore, for $\gamma<0$ the solution $\psi^{\gamma}(t)$ is globally defined on $(-\infty, 0]$, scatters backward in time   in $H^{1}(\R)$ and blows up in positive time.
\end{corollary}

Notice that  Corollary \ref{GroundPhase} shows that the results stated in Theorems \ref{ScatteaboveQ} and \ref{BlowupboveQ} are not symmetric in time. Finally, we have the following result.

\begin{corollary}\label{LongGeral}
Let $\mu\in \R\setminus\left\{0\right\}$, $u_{0}\in H^{1}(\R)$  with finite variance and let $u_{\mu}(t)$ be
the corresponding solution of \eqref{NLS} with initial data 
\[
u_{\mu, 0}(x)=e^{i\mu x^{2}}u_{0}(x).
\]
Assume that
\begin{equation}\label{LongCondi11}
	E(u_{0})[M(u_{0})]^{\sigma_{c}}\leq E(Q)[M(Q)]^{\sigma_{c}}.
\end{equation}
(i) If $[M(u_{0})]^{\sigma_{c}}N(u_{0})<[M(Q)]^{\sigma_{c}}N(Q)$, then for any $\mu>0$  the solution $u_{\mu}(t)$
scatters in $H^{1}(\R)$ forward in the time.\\
(ii) If $[M(u_{0})]^{\sigma_{c}}N(u_{0})>[M(Q)]^{\sigma_{c}}N(Q)$, then for any $\mu<0$  the solution $u_{\mu}(t)$
blows up in finite positive time.
\end{corollary}

We observe that under the conditions of Corollary \ref{LongGeral},
\[
E(u_{\mu, 0})=2\mu^{2}\|x u_{0}\|^{2}_{L^{2}}+2\mu \IM\int_{\R}x \partial_{x}u_{0}\overline{u_{0}}+E(u_{0}).
\]
Thus, $E(u_{\mu, 0})\to\infty$ as $\mu\to \infty$.

This paper is structured as follows. We fix notations at the end of Section \ref{sec:intro}. In Section \ref{S:preli},
we give some results that are necessary for later sections, including the smoothing properties of the integral equation.
In Section \ref{Criterion}, we prove the scattering criterion by a concentration-compactness-rigidity argument.
Then in Section \ref{S:2}, using the scattering criterion, we prove the Theorems \ref{ScattebelowQ} and \ref{ScatteaboveQ}. 
In Section \ref{DynaThre}, we establish the long time dynamics at mass and energy ground states threshold (Theorem \ref{DynaThre}).
Finally, in Section \ref{Blow-up-Corollary} we show the blow-up result of Theorem \ref{BlowupboveQ}. Moreover,
we prove the Corollaries \ref{GroundPhase} and \ref{LongGeral}.

\subsection*{Notations}
We write $A\lesssim B$ or $B\gtrsim A$  to signify $A\leq CB$ for some constant $C>0$. When  $A \lesssim B \lesssim A$ we write $A\sim B$.
For an interval $I\subset \R$, we use $L^{r}(I)$ to denote the Banach space of functions $f: I:\to \C$ such that the norm
\[
\|f\|_{L^{r}(I)}=\( \int_{\R}|f(x)|^{r}dx\)^{\frac{1}{r}},
\]
is finite, with the usual adjustments when $r=\infty$. 

For $s\in \R$ we introduce the Sobolev space
\[
H^{s}(\R)=\left\{f\in \Sch^{\prime}(\R), \|f\|_{H^{s}(\R)}
:=\|(1+|\xi|^{2})^{\frac{s}{2}}\hat{f}(\xi)\|_{L^{2}(\R)}\right\}
\]
and the homogeneous Sobolev space
\[
\dot{H}^{s}(\R)=\left\{f\in \Sch^{\prime}(\R), \|f\|_{\dot{H}^{s}(\R)}
:=\||\xi|^{s}\hat{f}(\xi)\|_{L^{2}(\R)}\right\},
\]
where $\hat{f}$ denote the Fourier transform. As usual, we denote $\R^{+}=(0, +\infty)$. Finally, we define the following indices
\[
a=2(p-1) \quad \text{and}\quad 
 b=\frac{2(p-1)}{p}.
\]

\section{Preliminaries}\label{S:preli}

In this section, we give preliminary results that will be used later. A function $u:\R \times I\to \C$ on a nonempty interval $[0, T_{+})$ is called a solution to the Cauchy problem \eqref{NLS}
if $u\in C_{t}H^{1}( \R\times K)\cap C_{t}H^{^{\frac{3}{4}}}((0, T_{+})\times\R)$ for every compact interval $K\subset (0, T_{+})$
and it satisfies the Duhamel formula
\begin{equation}\label{DuhamelF}
\begin{split}
u(x,t)&=e^{i\partial^{2}_{x}}u_{0}+i\int^{t}_{0}e^{i(t-s)\partial^{2}_{x}}\delta(x)|u(x,s)|^{p-1}u(x,s)ds\\
&=e^{i\partial^{2}_{x}}u_{0}+i\int^{t}_{0}\frac{e^{\frac{ix^{2}}{4(t-s)}}}{\sqrt{4\pi i (t-s)}}|u(0,s)|^{p-1}u(0,s)ds,
\end{split}
\end{equation}
for all $t\in [0, T_{+})$. 
Consider $f\in \dot{H}^{\gamma}$ with $\gamma\in \R$. For $t$, $s\in \R$ with $t\geq s$
we define the transformation 
\[
[\Pch_{s}f](x,t):=\int^{t}_{s}\frac{e^{\frac{ix^{2}}{4(t-\tau)}}}{\sqrt{4\pi i (t-\tau)}} f(\tau) d\tau.
\]
Moreover, for $t\in \R$ we define
\[
[\Sch f](x,t):=\int^{\infty}_{t}\frac{e^{\frac{ix^{2}}{4(t-\tau)}}}{\sqrt{4\pi i (t-\tau)}}f(\tau)  d\tau.
\]

To prove the scattering result, we we will need the following global-in-time estimates 
(see \cite[Proposition 2.1]{AdamiFukuHolmer} and \cite[Lemma 1]{ADAMI2001148}).

\begin{proposition}\label{Estimates-S}
Let $t$, $s$ and  $\gamma\in \R$. We have\\
(i) If $f\in \dot{H}^{\gamma}(\R)$, then
\[
\|[e^{i(t-s)\partial^{2}_{x}}f](0)\|_{\dot{H}^{\frac{2\gamma +1}{4}}_{t}}\lesssim \|f\|_{\dot{H}^{\gamma}}.
\]
(ii) If $f\in \dot{H}^{\frac{2\gamma-1}{4}}$ and $-\frac{1}{2}<\frac{2\gamma-1}{4}<\frac{1}{2}$, then
\[\begin{split}
\|[\Pch_{s}f](0, \cdot)\|_{\dot{H}^{\frac{2\gamma +1}{4}}_{t}}&\lesssim\|f\|_{\dot{H}^{\frac{2\gamma-1}{4}}},\\
\|[\Sch f](0, \cdot)\|_{\dot{H}^{\frac{2\gamma +1}{4}}_{t}}&\lesssim\|f\|_{\dot{H}^{\frac{2\gamma-1}{4}}}.
\end{split}\]
(iii) If $f\in \dot{H}^{\frac{2\gamma-1}{4}}$ and $-\frac{1}{2}<\frac{2\gamma-1}{4}<\frac{1}{2}$, then
\[\begin{split}
\|\Pch_{s}f\|_{L_{r}^{\infty}({\R})\dot{H}^{\gamma}_{x}}&\lesssim\|f\|_{\dot{H}^{\frac{2\gamma-1}{4}}},\\
\|\Sch f\|_{L_{r}^{\infty}({\R})\dot{H}^{\gamma}_{x}}&\lesssim\|f\|_{\dot{H}^{\frac{2\gamma-1}{4}}}.
\end{split}\]
\end{proposition}

\begin{remark}\label{BoundedGamma}
We recall that $\gamma_{c}=\frac{1}{2}-\frac{1}{p-1}$. Notice that if $p>3$, then
\[
\frac{1}{4}<\frac{2\gamma_{c}+1}{4}<\frac{1}{2} \quad \text{and}\quad 
 -\frac{1}{4}<\frac{2\gamma_{c}-1}{4}<0.
\]
Furthermore, by the Sobolev embedding $\dot{H}_{x}^{s}(\R)\hookrightarrow L^{r_{1}}_{x}(\R)$ with $\frac{1}{r_{1}}=\frac{1}{2}-s$, we get
\[
\|f\|_{L_{x}^{a}(\R)}\lesssim \|f\|_{\dot{H}_{x}^{\frac{2\gamma_{c}+1}{4}}}\quad \text{for} \quad a=2(p-1).
\]
And since $ L^{r^{\prime}_{1}}_{x}(\R) \hookrightarrow \dot{H}_{x}^{-s}(\R)$, it follows that
\[
\|f\|_{\dot{H}_{x}^{\frac{2\gamma_{c}-1}{4}}}\lesssim\|f\|_{L_{x}^{b}(\R)} \quad \text{for}\quad b=\frac{2(p-1)}{p}.
\]
\end{remark}

\subsection*{Variational Analysis}
We recall the following Gagliardo-Nirenberg inequality established by Holmer and Liu \cite[Proposition 1.3]{HOLMER2020123522},
\begin{equation}\label{Ga-Ni-Inequ}
N(u)\leq [K(u)]^{\frac{p+1}{4}} [M(u)]^{\frac{p+1}{4}} \quad \text{for all $u\in H^{1}(\R)$}.
\end{equation}
The ground state $Q$ optimizes the inequality \eqref{Ga-Ni-Inequ}, i.e.,
\[
N(Q)=[K(Q)]^{\frac{p+1}{4}} [M(Q)]^{\frac{p+1}{4}},
\]
where $Q$ is defined in \eqref{Groundsta}. Note also that $Q$ satisfies the Pohozaev identities:
\begin{equation}\label{PohoIdenti}
M(Q)=K(Q)=\frac{1}{2}N(Q).
\end{equation}
Thus,
\begin{equation}\label{EnergyIdentities}
E(Q)=\frac{(p-3)}{4(p+1)}N(Q)=\frac{(p-3)}{2(p+1)}K(Q).
\end{equation}
Moreover, 
\begin{equation}\label{OptConst}
2^{\sigma_{c}}=(K(Q)[M(Q)]^{\sigma_{c}})^{\frac{p+1}{4}}=\(\frac{1}{2}N(Q)[M(Q)]^{\sigma_{c}}\)^{\frac{p+1}{4}}.
\end{equation}

We also have the following coercivity property.

\begin{lemma}\label{Coercivity}
Let $p>3$ and $f\in H^{1}(\R)$ satisfy
\[
N(f)[M(f)]^{\sigma_{c}}\leq L<N(Q)[M(Q)]^{\sigma_{c}}
\]
for some positive constant $L$. Then there exists a positive constant $\eta=\eta(L, Q)$ such that
\begin{align}\label{IneG}
8K(f)-4N(f)&\geq \eta K(f),\\\label{IneEnergy}
E(f)&\geq \frac{\eta}{16} K(f).
\end{align}
\end{lemma}
\begin{proof}
We can write
\begin{equation}\label{Lident}
L=(1-r)N(Q)[M(Q)]^{\sigma_{c}}
\end{equation}
for some $r>0$ small. Notice that by Gagliardo-Nirenberg inequality \eqref{Ga-Ni-Inequ}
\[
[N(f)]^{\frac{(p+1)}{4}}\leq [K(f)]^{\frac{(p+1)}{4}}(N(f)[M(f)]^{\sigma_{c}})^{\frac{(p-3)}{4}}.
\]
Moreover, using the identities \eqref{OptConst} we have
\[
N(Q)[M(Q)]^{\sigma_{c}}=2^{\sigma_{c}}.
\]
Combining equations above and \eqref{Lident} we get
\[\begin{split}
[N(f)]^{\frac{(p+1)}{4}}\leq 2^{\frac{(p+1)}{4}}[K(f)]^{\frac{(p+1)}{4}}
\( \frac{N(f)[M(f)]^{\sigma_{c}}}{N(Q)[M(Q)]^{\sigma_{c}}}\)^{\frac{(p-3)}{4}}\\
\leq 2^{\frac{(p+1)}{4}} [K(f)]^{\frac{(p+1)}{4}} (1-r)^{\frac{(p-3)}{4}}.
\end{split}\]
Then,
\[
8K(f)-4N(f)\geq 8K(f)-8K(f)(1-r)^{\frac{1}{\sigma_{c}}}=\eta K(f),
\]
where $\eta=8(1-(1-r)^{\frac{1}{\sigma_{c}}})>0$. This proves the estimate \eqref{IneG}.
Finally, since $p> 3$, from \eqref{IneG} we obtain
\[
E(f)=\frac{1}{2}\(K(f)-\frac{1}{2}N(f)\)+\frac{(p-3)}{4(p+1)}N(f)\\
\geq \frac{1}{2}\(K(f)-\frac{1}{2}N(f)\)\geq \frac{\eta}{16}K(f).
\]
This completes the proof of lemma.
\end{proof}

\section{Scattering criterion}\label{Criterion}

In this section, we show the Theorem \ref{Th1}.

\subsection*{Cauchy problem}

In the following result we have a sufficiently condition for scattering. See \cite[Proposition 2.4]{AdamiFukuHolmer} for the proof.
\begin{proposition}\label{Condition-scatte}
Let $p\geq3$, $u_{0}\in H^{1}(\R)$ and $u(t)$ be the corresponding solution of Cauchy problem \eqref{NLS}.
If $u$ is forward global, uniformly bounded in $H_{x}^{1}(\R)$ and $\|u(0, \cdot)\|_{L^{a}_{t}(\R^{+})}< \infty$,
then $u(t)$ scatters forward in time.
\end{proposition}

We recall a small data global existence result for the equation \ref{NLS}.

\begin{proposition}\label{smalldataexistence}
Let $p\geq3$. There exist $0<\delta_{sd}\leq1$ and a positive constant $C_{sd}$ such that if
$\|u_{0}\|_{\dot{H}^{\gamma_{c}}_{x}(\R)}\leq \delta_{sd}$, then the solution of \eqref{NLS} with initial data $u_{0}$ 
is global in $\dot{H}^{\gamma_{c}}(\R)$  and
\[
\|u(0, \cdot)\|_{L^{a}_{t}(\R^{+})}+\|u\|_{L^{\infty}_{t}(\R^{+}, {H}^{1}(\R))}\leq C_{sd}\|u_{0}\|_{{H}^{1}_{x}(\R)}.
\]
\end{proposition}
For the proof of Proposition \ref{smalldataexistence}, see \cite[Propositions 2.3 and 2.4]{AdamiFukuHolmer}.

\begin{remark}[Existence of wave operators]\label{Exitencescat}
Using the same argument developed in the proof of Lemma 4.2 in \cite{AdamiFukuHolmer}, we can show that given $\psi\in H^{1}(\R)$,
there exists $T\in \R$ and a solution $v\in C([T,\infty), H^{1}(\R))$ to \eqref{NLS} such that
\[
\|v(t)-e^{it\partial^{2}_{x}}\psi\|_{H^{1}}\rightarrow 0 \quad\text{ as $t\to\infty$}.
\]
\end{remark}

\subsection*{Perturbation lemma and linear profile decomposition}
We will use a perturbation result and a lemma for linear profiles.

For the proof of the following result we refer the reader to \cite[Proposition 2.5]{AdamiFukuHolmer}.

\begin{proposition}\label{Perturbationlemma}
Let $p\geq 3$.  For any $M\gg 1$, there exist $\epsilon=\epsilon(M)\ll 1$ and $C=C(M)>0$
such the following holds. If $u\in C([0,\infty), H^{1}(\R))$ is a solution to \eqref{NLS} and if
$\tilde{u}\in C([0,\infty), H^{1}(\R))$ is a solution of the equation with source term $e$:
\[
i\partial_{t}\tilde{u}+\partial^{2}_{x}\tilde{u}+\delta(x)|\tilde{u}|^{p-1}\tilde{u}=\delta(x)e
\]
with  
\[
\|e(0, \cdot)\|_{L_{t}^{b}(\R^{+})}<\epsilon, \quad \|\tilde{u}(0, \cdot)\|_{L_{t}^{a}(\R^{+})}\leq M
\]
and if
\[
\|[e^{i(t-t_{0})\partial^{2}_{x}}(u(t_{0})-\tilde{u}(t_{0}))](0)\|_{L_{t}^{a}((t_{0}, \infty))}\leq \epsilon
\]
for some $t_{0}\geq 0$, then
\[
\|u(0, \cdot)\|_{L_{t}^{a}(\R^{+})}\leq C(M).
\]
\end{proposition}

Using Proposition \ref{Estimates-S}, the proof of the following result is an easy modification of arguments 
used in \cite[Proposition 2.5]{AdamiFukuHolmer}.

\begin{proposition}\label{pertucompact}
Let $p\geq 3$ and $[0, T]$ an compact interval and let $\tilde{u}:[0,T]\times \R \to \C$ be a solution to 
\[
i\partial_{t}\tilde{u}+\partial^{2}_{x}\tilde{u}+\delta(x)|\tilde{u}|^{p-1}\tilde{u}=\delta e
\]
for some source term $e$. Assume that there exists $M>0$ such that 
\[
\|\tilde{u}(0, \cdot)\|_{L_{t}^{a}([0,T]) \cap L_{t}^{\infty}([0,T])}\leq M.
\] 
Moreover, suppose that $u:[0,T]\times \R \to \C$  is solution to \eqref{NLS}. Assume that we have the following smallness conditions
\[\begin{split}
&\|[e^{it\partial^{2}_{x}}(u(0)-\tilde{u}(0))](0)\|_{L_{t}^{a}([0,T]) \cap L_{t}^{\infty}([0,T])}\leq \epsilon,\\
&\|e(0, \cdot)\|_{L_{t}^{b}([0,T]) \cap L_{t}^{\infty}([0,T])}<\epsilon,
\end{split}\]
for some $0<\epsilon<\epsilon_{1}$ with $\epsilon_{1}=\epsilon_{1}(M)$ small constant. Then 
\[
\|[(\tilde{u}-u)](0,\cdot)\|_{L_{t}^{a}([0,T]) \cap L_{t}^{\infty}([0,T])}\leq \epsilon \, C(M,T).
\]
\end{proposition}

We need the following linear profile decomposition, which is a key ingredient.

\begin{proposition}[Linear profile decomposition]\label{Linear-profile-lemma}
Let $p> 3$.  Let $\left\{\phi_{n}\right\}_{n\geq 1}$ be a bounded sequence of $H^{1}(\R)$.
Then for each integer $J\geq 1$, there exists a subsequence, which we still denote by $\left\{\phi_{n}\right\}_{n\geq 1}$,
 and \\
(i) for each $1\leq j \leq J$, there exists a fixed profile $\psi^{j}\in H^{1}(\R)$;\\
(ii) for each $1\leq j \leq J$, there exits a sequence of time shifts $\left\{t^{j}_{n}\right\}_{n\geq 1}\subset \R$; \\
(iii) for each $1\leq j \leq J$, there exits a sequence of remainders $\left\{W^{J}_{n}\right\}_{n\geq 1}\subset H^{1}(\R)$
such that we can write
\[
\phi_{n}=\sum^{J}_{j=1}e^{-it^{j}_{n}\partial^{2}_{x}}\psi^{j}+W^{J}_{n},
\]
and the following hold.
\begin{itemize}[leftmargin=5mm]
		\item Orthogonality of the parameters:
\begin{equation}\label{Ortho}
	\lim_{n\to\infty}|t^{i}_{n}-t^{j}_{n}|=\infty, \quad \text{for $1\leq i\neq j\leq J$}.
\end{equation}
    \item Asymptotic smallness of the reminder:
		\begin{equation}\label{AsySmall}
\lim_{J\to\infty}\( \lim_{n\to\infty} \| [e^{it\partial^{2}_{x}}W^{J}_{n}](0)\|_{L^{a}_{t}(\R)\cap L^{\infty}_{t}(\R)}\)=0.
\end{equation}
		\item  Orthogonality in norms: for fixed $J$ and any $\gamma\in [0,1]$,
		\begin{equation}\label{Normsortho}
	\|  \phi_{n} \|^{2}_{\dot{H}^{\gamma}}=\sum^{J}_{j=1}\| \psi^{j} \|^{2}_{\dot{H}^{\gamma}}+
		\| W^{J}_{n} \|^{2}_{\dot{H}^{\gamma}}+o_{n}(1).
\end{equation}
\item Asymptotic Pythagorean expansion: for fixed $J$, 
		\begin{equation}\label{PythagoreanEx}
		E(\phi_{n})=\sum^{J}_{j=1}E(e^{-it^{j}_{n}\partial^{2}_{x}}\psi^{j})+E(W^{J}_{n})+o_{n}(1).
\end{equation}
		\end{itemize}
\end{proposition}
\begin{proof}
We follow the same spirit as in the proof of \cite[Proposition 3.1]{AdamiFukuHolmer}. 
Let $A=\limsup_{n\to \infty}\|\phi_{n}\|_{H_{x}^{1}(\R)}$. For $R>0$,
we fix a real-valued, symmetric function $\zeta_{R}\in C^{\infty}_{c}(\R)$ such that $0\leq \hat{\zeta}_{R}\leq 1$, $\hat{\zeta}_{R}(\xi)=1$
for $\frac{1}{R}\leq |\xi|\leq R$ and supported in $\frac{1}{2R}\leq |\xi|\leq 2R$.  
Let $C_{1}:=\limsup_{n\rightarrow \infty}\| [e^{it\partial^{2}_{x}}\phi_{n}](0)\|_{L^{\infty}_{t}(\R)}$ 
and
$B_{1}:=\limsup_{n\rightarrow \infty}\| [e^{it\partial^{2}_{x}}\phi_{n}](0)\|_{L^{a}_{t}(\R)}$. 
We have four cases: (i) $C_{1}=0$ and $B_{1}=0$; (ii) $C_{1}>0$ and $B_{1}=0$; (iii) $C_{1}=0$ and $B_{1}>0$ and 
(iv) $C_{1}>0$ and $B_{1}>0$. We only deal with the case $C_{1}>0$ and $B_{1}>0$. 
The proof in the other cases is similar.

Suppose that
$C_{1}>0$ and $B_{1}>0$. Pass to a subsequence, we may assume 
$\lim_{n\rightarrow \infty}\| [e^{it\partial^{2}_{x}}\phi_{n}](0)\|_{L^{\infty}_{t}(\R)}=C_{1}$. In particular, passing to a subsequence if necessary, we have that
\[
\sup_{t\in \R}|[e^{it\partial^{2}_{x}}\phi_{n}](0)|\geq \frac{C_{1}}{2}.
\]
It follows that there exists a sequence of times $\left\{t^{1}_{n}\right\}_{n\geq1}$
such that
\begin{equation}\label{inflimit}
|[e^{it^{1}_{n}\partial^{2}_{x}}\phi_{n}](0)|\geq \frac{C_{1}}{2},
\end{equation}
for all $n\geq 1$. Since $\|e^{it^{1}_{n}\partial^{2}_{x}}\phi_{n}\|_{H^{1}}\lesssim 1$ for all $n$, there exists $\psi^{1}$ such that
$e^{it^{1}_{n}\partial^{2}_{x}}\phi_{n}\rightharpoonup \psi^{1}$ in $H^{1}(\R)$ as $n \to \infty$. By using the compact embedding 
$H^{1}[-1, 1]\hookrightarrow C[-1, 1]$ we infer that $[e^{it^{1}_{n}\partial^{2}_{x}}\phi_{n}](0)\to \psi^{1}(0)$. Then, from 
\eqref{inflimit} we get $|\psi^{1}(0)|\geq \frac{C_{1}}{2}$. In particular, $\psi^{1}\neq 0$ and by the  inequality 
$2|\psi^{1}(0)|^{2}\leq \|\psi^{1}\|^{2}_{H^{1}}$we see that
\begin{equation}\label{SobInf}
\|\psi^{1}\|^{2}_{H^{1}}\geq \frac{(C_{1})^{^{2}}}{2}.
\end{equation}
We set $W^{1}_{n}=\phi_{n}-e^{-it^{1}_{n}\partial^{2}_{x}}\psi^{1}$. Then we obtain that for $0\leq \gamma \leq 1$,
\[
\lim_{n \to \infty}\|W^{1}_{n}\|^{2}_{\dot{H}^{\gamma}}=\lim_{n \to \infty}\|\phi_{n}\|^{2}_{\dot{H}^{\gamma}}
-\|\psi^{1}\|^{2}_{\dot{H}^{\gamma}}.
\]
In particular, $\limsup_{n \to \infty}\|W^{1}_{n}\|^{2}_{{H}^{1}}<\infty$. 
Now, by translation invariance of $L^{a}_{t}(\R)$-norm, the argument developed in \cite[Proposition 3.1]{AdamiFukuHolmer}  shows that by choosing $R_{1}=\left\langle 2AB^{-1}_{1}\right\rangle^{\max\left\{\frac{1}{\gamma_{c}}, \frac{1}{1-\gamma_{c}}\right\}}$ (note that $B_{1}>0$), we have
\[
\begin{split}
\(\frac{1}{2}B_{1}\)^{\frac{a}{a-2}}&=\frac{1}{2^{\frac{a}{a-2}}}\lim_{n\to\infty}\|[e^{i(t+t^{1}_{n})\partial^{2}_{x}}\phi_{n}](0)\|^{\frac{a}{a-2}}_{L^{a}_{t}(\R)}\\
&\leq [A^{2}R_{1}]^{\frac{1}{a-2}}
\lim_{n \to \infty}\| [\zeta_{R_{1}}\ast e^{i(t+t^{1}_{n})\partial^{2}_{x}}\phi_{n}](0)\|_{L^{\infty}_{t}(\R)}.
\end{split}
\]
Since $e^{it\partial^{2}_{x}}$ commutes with the convolution with $\zeta_{R_{1}}$,  by the weak convergence we infer that
\[
\lim_{n \to \infty}\| [\zeta_{R_{1}}\ast e^{i(t+t^{1}_{n})\partial^{2}_{x}}\phi_{n}](0)\|_{L^{\infty}_{t}(\R)}
=\| [\zeta_{R_{1}}\ast e^{it\partial^{2}_{x}}\psi^{1}](0)\|_{L^{\infty}_{t}(\R)}.
\]
Moreover, by plancherel's formula we see that
\[
\| [\zeta_{R_{1}}\ast e^{it\partial^{2}_{x}}\psi^{1}](0)\|_{L^{\infty}_{t}(\R)}\leq C_{\gamma_{c}
}R_{1}^{\frac{1-2\gamma_{c}}{2}}\|\psi^{1}\|_{\dot{H}^{\gamma_{c}}}.
\]
Thus, we obtain
\begin{equation*}
\|\psi^{1}\|_{\dot{H}^{\gamma_{c}}}\geq 
[{C_{\gamma_{c}}}]^{-1}\( \frac{B_{1}}{2} \)^{\frac{a}{a-2}} A^{\frac{2}{2-a}}R^{-\theta}_{1},
\end{equation*}
where $\theta=\frac{a}{a-2}+\frac{1-2\gamma_{c}}{2}$.

Next, we obtain the functions $\psi^{j}$, for all $j\geq 2$ inductively (see, for example \cite[Lemma 5.2]{HolmerRoudenko2008}). Indeed, we construct a 
sequence $\left\{t^{j}_{n}\right\}_{n\geq1}$ and a profile $\psi^{j}$ such that 
\begin{align}\label{SobInf11}
	\|\psi^{j}\|^{2}_{H^{1}}&\geq \frac{(C_{J})^{^{2}}}{2},\\\label{InequalPor}
	\|\psi^{j}\|_{\dot{H}^{\gamma_{c}}}&\geq 
[C_{\gamma_{c}}]^{-1}\( \frac{B_{J}}{2} \)^{\frac{a}{a-2}}A^{\frac{2}{2-a}}R_{J}^{-\theta},
\end{align}
where
\begin{align*}
C_{J}:=\limsup_{n\rightarrow \infty}\| [e^{it\partial^{2}_{x}}W^{J-1}_{n}](0)\|_{L^{\infty}_{t}(\R)},\\
B_{J}:=\limsup_{n\rightarrow \infty}\| [e^{it\partial^{2}_{x}}W^{J-1}_{n}](0)\|_{L^{a}_{t}(\R)}.
\end{align*}
But then, by \eqref{SobInf11} and \eqref{InequalPor} we infer 
\[
\sum^{\infty}_{J=1}\frac{(C_{J})^{^{2}}}{2}\leq \lim_{J \to \infty}\sum^{J}_{j=1}\|\psi^{j}\|^{2}_{H^{1}}
\leq \lim_{n \to \infty}\|\phi_{n}\|^{2}_{H^{1}}\lesssim 1.
\]
and
\[
\sum^{\infty}_{J=1} {B_{J}}^{\frac{a}{a-2}}R_{J}^{-\theta} \lesssim \lim_{J \to \infty}\sum^{J}_{j=1}\|\psi^{j}\|_{H^{\gamma_{c}}}
\lesssim 1.
\]
Therefore $C_{J}\to 0$ and $B_{J}\to 0$ as $n\to \infty$. Here we have used that $\theta>0$.

The remainder of the proof is similar to that of  
\cite[Proposition 3.1]{AdamiFukuHolmer}. This completes the proof of proposition.
\end{proof}

\subsection*{Scattering criterion}
Suppose that $u(t)$ is a solution of Cauchy problem \eqref{NLS} with initial data $u_{0}\in H^{1}(\R)$ 
satisfying \eqref{InequSc}. Notice that by the conservation of energy and assumption \eqref{InequSc}
we see that
\[
\sup_{t\in [0,T_{+})}K(u(t))\leq C(E(u_{0}), Q),
\]
which implies that the solution is global (i.e., $T_{+}=\infty$). Note also that to get scattering criterion in Theorem \ref{Th1}, from Proposition \ref{Condition-scatte}, it is enough to get $\|u(0, \cdot)\|_{L^{a}_{t}(\R^{+})}<\infty$. 

Let $L>0$ and $\delta>0$. For $u(t)$ satisfying 
\begin{equation}\label{ConditionsimpliSc}
\sup_{t\in [0,\infty)}N(u(t))[M(u(t))]^{\sigma_{c}}\leq L, \quad E(u(t))[M(u(t))]^{\sigma_{c}}\leq \delta,
\end{equation}
with $L<N(Q)[M(Q)]^{\sigma_{c}}$, we define
\[
S(L,\delta):=\sup\left\{  \|u(0, \cdot)\|_{L^{a}_{t}(\R^{+})}: \,\,\ 
\text{$u(t)$ satisfies \eqref{ConditionsimpliSc}}\right\}.
\]
By \eqref{IneEnergy} we infer that $E(u_{0})\geq 0$. Thus, by interpolation, \eqref{IneEnergy} and \eqref{ConditionsimpliSc} we obtain
\[
\|  u_{0} \|^{\frac{2}{\gamma_{c}}}_{\dot{H}^{\gamma_{c}}}\leq K(u_{0})[M(u_{0})]^{\sigma_{c}}\leq \frac{16}{\eta} \delta.
\]
By inequality above,  Propositions \ref{Condition-scatte} and \ref{smalldataexistence}, we have that $S(L,\delta)<\infty$ for $\delta$ small enough. 

Suppose by contradiction that Theorem \ref{Th1} fails. Then $S(L,\delta)=\infty$ for some $\delta<\infty$. As $S(L,\delta)<\infty$ 
for $\delta \ll 1$, using the monotonicity of $S(L,\delta)$, there exists a critical level $0<\delta_{c}(L)<\infty$ such that
\begin{equation}\label{criticald}
\delta_{c}:=\delta_{c}(L)=\sup\left\{\delta: S(L,\delta)<\infty\right\}=\inf\left\{\delta: S(L,\delta)=\infty\right\}.
\end{equation}
Notice that $S(L,\delta_{c})=\infty$. By definition, this implies that  there exists a sequence of initial data $u_{n}(0)$
such that
\begin{equation}\label{sequencecritical}
\begin{split}
&\sup_{t\in [0,\infty)}N(u_{n}(t))[M(u_{n}(t))]^{\sigma_{c}}\leq L, \quad E(u_{n}(t))[M(u_{n}(t))]^{\sigma_{c}}\searrow \delta_{c},\\
&\|u_{n}(0, \cdot)\|_{L^{a}_{t}(\R^{+})}=\infty \quad \text{for all $n$,}
\end{split}
\end{equation}
where $u_{n}(t)$ is a global solution (i.e., $T_{+}=\infty$) to \eqref{NLS} with initial data $u_{n}(0)$.
Now, our goal is to prove the existence of a critical element $u_{c}(t)$ such that
$E(u_{c}(t))[M(u_{c}(t))]^{\sigma_{c}}=\delta_{c}$, $\sup_{t\in [0,\infty)}N(u_{c}(t))[M(u_{c}(t))]^{\sigma_{c}}=L$
and \[\|u_{c}(0, \cdot)\|_{L^{a}_{t}(\R^{+})}=\infty.\] More precisely, we have the following result.

\begin{proposition}[Existence and compactness of critical element]\label{Compactness-critical}
There exists $u_{c}(0)\in H^{1}(\R)$ such that if $u_{c}(t)$ is the corresponding solution of \eqref{NLS} 
with initial data $u_{c}(0)$, then  $u_{c}(t)$ satisfies
\begin{equation}\label{criticalsolution}
\begin{split}
&M(u_{c}(t))=1, \quad\sup_{t\in [0,\infty)}N(u_{c}(t))=L, \\
& E(u_{c}(t))=\delta_{c},\quad \|u_{c}(0, \cdot)\|_{L^{a}_{t}(\R^{+})}=\infty
\end{split}
\end{equation}
and $\Omega:=\left\{u_{c}(t): t\geq 0 \right\}$ is precompact in $H^{1}(\R)$.
\end{proposition}
\begin{proof}
We observe that the quantities $E(u)[M(u)]^{\sigma_{c}}$ and $\sup_{t\geq 0}N(u(t))[M(u(t))]^{\sigma_{c}}$ 
are both invariant under the scaling \eqref{Scale-inv}. Thus, since the equation \eqref{NLS} also is invariant under
\eqref{Scale-inv}, we can assume that 
\begin{equation}\label{criticalscaling}
\begin{split}
&M(u_{n}(0))=1, \quad\sup_{t\in [0,\infty)}N(u_{n}(t))\leq L, \\
& E(u_{n}(0))=\delta_{c},\quad \|u_{n}(0, \cdot)\|_{L^{a}_{t}(\R^{+})}=\infty,
\end{split}
\end{equation}
and we may apply the profile decomposition to $\varphi_{n}:=u_{n}(0)$. Therefore, by Proposition \ref{Linear-profile-lemma} we write
\begin{equation}\label{initialprimer}
\varphi_{n}=\sum^{J}_{j=1}e^{-it^{j}_{n}\partial^{2}_{x}}\psi^{j}+W^{J}_{n}
\end{equation}
for all $n\geq1$, where the sequences satisfy properties \eqref{Ortho}-\eqref{PythagoreanEx}. Moreover, 
we may assume that either $t^{j}_{n}=0$ or $|t^{j}_{n}|\rightarrow \infty$.

Define the nonlinear profile $v^{j}$ associated to $\psi^{j}$ in the following way:\\
(i) If $t^{j}_{n}=0$, then $v^{j}$ is the maximal solution to equation \eqref{NLS} with initial data $v^{j}(0)=\psi^{j}$;\\
(ii) If $t^{j}_{n}\rightarrow \infty$, then $v^{j}$ is the maximal solution to equation \eqref{NLS} that scatters
backward in time to $e^{it\partial^{2}_{x}}\psi^{j}$, which existence is guaranteed by Remark \ref{Exitencescat}. 
In particular,
\[
\lim_{n\to\infty}\|v^{j}(-t^{j}_{n})-e^{-it^{j}_{n}\partial^{2}_{x}}\psi^{j}\|_{H^{1}}=0.
\]
(iii) Similarly,  if $t^{j}_{n}\rightarrow -\infty$, then $v^{j}$ is the maximal solution to equation \eqref{NLS} that scatters
forward in time to $e^{it\partial^{2}_{x}}\psi^{j}$. In particular, 
\[
\lim_{n\to\infty}\|v^{j}(-t^{j}_{n})-e^{-it^{j}_{n}\partial^{2}_{x}}\psi^{j}\|_{H^{1}}=0.
\]
Let $v^{j}_{n}(t):=v^{j}(t-t^{j}_{n})$. This is still solution of equation \eqref{NLS} and satisfies
\begin{equation}\label{Limone}
\lim_{n\to\infty}\|v_{n}^{j}(0)-e^{-it^{j}_{n}\partial^{2}_{x}}\psi^{j}\|_{H^{1}}=0.
\end{equation}
Now, we rewrite \eqref{initialprimer} as 
\begin{equation}\label{second-initial}
\varphi_{n}(x)=\sum^{J}_{j=1}v^{j}_{n}(x,0)+\tilde{W}^{J}_{n}(x),
\end{equation}
for $n\geq1$,  where
\begin{equation}\label{Residual}
\tilde{W}^{J}_{n}(x)=\sum^{J}_{j=1}[e^{-it^{j}_{n}\partial^{2}_{x}}\psi^{j}(x)-v^{j}_{n}(x,0)]+{W}^{J}_{n}(x).
\end{equation}
From Remark \ref{BoundedGamma}, Sobolev inequality and Proposition \ref{Estimates-S} (i) we infer that
\begin{align*}
	&\|[e^{it\partial^{2}_{x}} \tilde{W}^{J}_{n}](0) \|_{L^{a}_{t}(\R)\cap L^{\infty}_{t}(\R)}\\
	& \leq \sum^{J}_{j=1}\|e^{it\partial^{2}_{x}} [e^{-it^{j}_{n}\partial^{2}_{x}}\psi^{j}-v^{j}_{n}(0, \cdot)](0) \|_{L^{a}_{t}(\R)\cap L^{\infty}_{t}(\R) }
	+\|[e^{it\partial^{2}_{x}} {W}^{J}_{n}](0) \|_{L^{a}_{t}(\R)\cap L^{\infty}_{t}(\R)}\\
	&\leq C\sum^{J}_{j=1}\|[e^{-it^{j}_{n}\partial^{2}_{x}}\psi^{j}-v^{j}_{n}](0)\|_{{H}^{1}_{x}}+
	\|[e^{it\partial^{2}_{x}} {W}^{J}_{n}](0) \|_{L^{a}_{t}(\R)\cap L^{\infty}_{t}(\R)}.
\end{align*}
Using \eqref{AsySmall} and \eqref{Limone} we get
\begin{equation}\label{limnew}
\lim_{J\to\infty}\(\lim_{n\to\infty}
\|[e^{it\partial^{2}_{x}} \tilde{W}^{J}_{n}](0) \|_{L^{a}_{t}(\R)\cap L^{\infty}_{t}(\R)}\)=0.
\end{equation}

Note that by \eqref{Normsortho},  there exists $J_{\ast}\in \N$ such that  $\|\psi^{j}\|_{H^{1}}\leq \delta_{sd}$ 
for $j\geq J_{\ast}$. Therefore, from \eqref{Limone} and using Propositions
\ref{smalldataexistence} and \ref{Condition-scatte} we infer that for $j\geq J_{\ast}$ the solutions $v^{j}_{n}(t)$ to \eqref{NLS} are global and 
\[
\| v^{j}_{n}(0,\cdot) \|_{L^{a}_{t}(\R^{+})}+\| v^{j}_{n}(0,\cdot) \|_{L^{\infty}_{t}(\R^{+})}\lesssim
\|\psi^{j}\|_{H^{1}}
\]
for all $j\geq J_{\ast}$. 
We have the following \\
\textit{Claim 1.} There exists $1\leq j_{0}<J_{\ast}$ such that
\begin{equation}\label{limitVClaim}
\limsup_{n\to\infty}\| v^{j_{0}}_{n}(0,\cdot) \|_{L^{a}_{t}(\R^{+})}=\infty.
\end{equation}
Indeed, assume by contradiction that for all $1\leq j<J_{\ast}$, 
\[
\limsup_{n\to\infty}\| v^{j}_{n}(0,\cdot) \|_{L^{a}_{t}(\R^{+})}<\infty. 
\]
Then, there exists $C>0$ such that for $n$ big enough  $\| v^{j}_{n}(0,\cdot) \|_{L^{a}_{t}(\R^{+})}\leq C$.
With the same argument developed in \cite[Section 4]{AdamiFukuHolmer} and using the Lemma \ref{Perturbationlemma}
we can find that for $n$ sufficiently large $\| u_{n}(0,\cdot) \|_{L^{a}_{t}(\R^{+})}<\infty$, which is a contradiction.
This proves the Claim 1.

Next, by reordering, we can choose $1\leq J^{1}_{\ast}\leq J_{\ast}$ such that
\begin{equation}\label{DefJi}
\begin{split}
\limsup_{n\to\infty}\| v^{j}_{n}(0,\cdot) \|_{L^{a}_{t}(\R^{+})}&=\infty \quad \text{for $1\leq j\leq J^{1}_{\ast}$,}\\
\limsup_{n\to\infty}\| v^{j}_{n}(0,\cdot) \|_{L^{a}_{t}(\R^{+})}&<\infty \quad \text{for $j>J^{1}_{\ast}$.}
\end{split}
\end{equation}
Following \cite[Proposition 5.6]{KiiVisan2008}, for each $m$, $n \in \N$, we define $j(m,n)\in \left\{1,2, \ldots, J^{1}_{\ast}\right\}$ 
and a compact interval $I^{m}_{n}$ of the form $[0, T]$ such that
\begin{equation}\label{CompactI}
\sup_{1\leq j \leq J^{1}_{\ast}}\| v^{j}_{n}(0,\cdot) \|_{L^{a}_{t}(I^{m}_{n})}
=\| v^{j(m,n)}_{n}(0,\cdot) \|_{L^{a}_{t}(I^{m}_{n})}=m.
\end{equation}
Thus, since $J^{1}_{\ast}<\infty$, using the pigeonhole principle, there is a $1\leq j_{1}\leq  J^{1}_{\ast}$ such that
 $j(m,n)=j_{1}$ for infinitely many $m$ and for infinitely many $n$.

There are two scenarios to consider.

\textbf{Scenario 1:} More than one $\psi^{j}\neq 0$.  By \eqref{Normsortho} and \eqref{PythagoreanEx} we infer that
\[
E(v^{j_{1}}_{n}(t))[M(v^{j_{1}}_{n}(t))]^{\sigma_{c}}<\delta_{c},
\]
where $j_{1}$ is defined above. Here we recall that $\delta_{c}$  is the critical  level.
Using \eqref{CompactI} and \eqref{criticald}, by definition of $v^{j_{1}}_{n}(t)$  we have
\begin{equation}\label{Lcontra}
\limsup_{m\to\infty}\limsup_{n\to\infty}\sup_{t\in I^{m}_{n}}N(v^{j_{1}}_{n}(t))[M(v^{j_{1}}_{n}(t))]^{\sigma_{c}}\geq L.
\end{equation}
By reordering we can choose $j_{1}=1$. Now, we need the following result.
We denote $\text{NLS}(t)\phi$ the solution to equation \eqref{NLS} with initial data $\phi$. We recall that $\varphi_{n}=u_{n}(0)$.

\begin{lemma}\label{ExpanBoundedNLS}
Let $T>0$ fixed and assume that the solution $u_{n}(t)=\text{NLS}(t)\varphi_{n}$  exists up to time $T$ for all $n\geq1$, and 
\begin{equation}\label{KunBounded}
\limsup_{n\to\infty}\sup_{t\in [0,T]}K(u_{n}(t))<\infty.
\end{equation}
 Then for all $1\leq j\leq J$, the nonlinear profiles $v_{n}^{j}(t)$ exist up to time $T$ and  
\begin{align}\label{DecomK}
	K(u_{n}(t))&=\sum^{J}_{j=1}K(v_{n}^{j}(t))+K(\tilde{W}^{J}_{n}(t))+o_{n,J}(1), \\\label{DecomNV}
	N(u_{n}(t))&=\sum^{J}_{j=1}N(v_{n}^{j}(t))+N(\tilde{W}^{J}_{n}(t))+o_{n,J}(1) 
\end{align}
for $t\in [0,T]$, 
where $o_{n,J}\rightarrow 0$ as $n$, $J\to \infty$ uniformly on $0\leq t\leq T$. Here $\tilde{W}^{J}_{n}(t)=\text{NLS}(t)\tilde{W}^{J}_{n}$, where $\tilde{W}^{J}_{n}$ is the remainder given in \eqref{second-initial}.
\end{lemma}
Let us assume, for a moment, that Lemma \ref{ExpanBoundedNLS} is true. We recall that
\[
1=M(u_{n}(t))\geq M(v^{j}_{n}(t)).
\]
Since more that one $\psi^{j}\neq 0$, combining \eqref{criticalscaling}, \eqref{Lcontra} 
and \eqref{DecomNV} we obtain
\[\begin{split}
L&\geq \limsup_{n\to\infty}\sup_{t\in [0,\infty)}N(u_{n}(t))[M(u_{n}(t))]^{\sigma_{c}}\\
&\geq
\lim_{J\to\infty}\limsup_{m\to\infty}\limsup_{n\to\infty}\sup_{t\in I^{m}_{n}}\sum^{J}_{j=1}N(v^{j}_{n}(t))[M(v^{j}_{n}(t))]^{\sigma_{c}}\\
&>\limsup_{m\to\infty}\limsup_{n\to\infty}\sup_{t\in I^{m}_{n}}N(v^{1}_{n}(t))[M(v^{1}_{n}(t))]^{\sigma_{c}}\geq L,
\end{split}\]
which is a contradiction.

\textbf{Scenario 2:} $\psi^{1}\neq 0$ and $\psi^{j}=0$ for every $j\geq2$. In this case we have
\[
\phi_{n}=e^{-it^{j}_{n}\partial^{2}_{x}}\psi^{1}+W^{1}_{n}, \quad 
\lim_{n\to\infty}\|[e^{it\partial^{2}_{x}} {W}^{1}_{n}](0) \|_{L^{a}_{t}(\R)}=0.
\]
If $t^{j}_{n}\rightarrow -\infty$, we obtain
\[
\| [e^{it\partial^{2}_{x}} \phi_{n}](0) \|_{L^{a}_{t}(\R^{+})}\leq 
\| [e^{it\partial^{2}_{x}} \psi^{1}](0) \|_{L^{a}_{t}([-t^{1}_{n},\infty))}+
\|[e^{it\partial^{2}_{x}} {W}^{1}_{n}](0) \|_{L^{a}_{t}(\R)}\rightarrow 0
\]
as $n\to \infty$. By inequality above, using a standard continuity argument and Proposition \ref{smalldataexistence}  we may shows that
$\|u_{n}(0, \cdot)\|_{L^{a}_{t}(\R^{+})}<\infty$ for $n$ big enough, which is a contradiction. A similar claim is valid for  
$t^{j}_{n}\rightarrow \infty$.  Thus, by \eqref{second-initial} we can write
\[
\phi_{n}(x)=v^{1}(x,0)+\tilde{W}^{1}_{n}(x).
\]
Notice that by \textit{Claim 1} above $\|v^{1}(0, \cdot)\|_{L^{a}_{t}(\R^{+})}=\infty$.
Moreover, from \eqref{Normsortho} and \eqref{PythagoreanEx} we can show that
\[
M(v^{1}(t))\leq 1, \quad E(v^{1}(t))\leq \delta_{c}.
\] 
Notice also that from \eqref{Lcontra} we have that
\[
\sup_{t\in [0,\infty)}N(v^{1}(t))[M(v^{{1}}(t))]^{\sigma_{c}}=\limsup_{m\to\infty}\sup_{t\in I^{m}_{1}}N(v^{1}(t))[M(v^{{1}}(t))]^{\sigma_{c}}\geq L,
\]
and, by \eqref{DecomNV},  
\[\begin{split}
L&\geq \limsup_{n\to\infty}\sup_{t\in [0,\infty)}N(u_{n}(t))[M(u_{n}(t))]^{\sigma_{c}}\\
&\geq\limsup_{m\to\infty}\sup_{t\in I^{m}_{1}}N(v^{1}(t))[M(v^{{1}}(t))]^{\sigma_{c}}\geq L.
\end{split}
\]
Hence, we get
\begin{equation}\label{LequalV}
\sup_{t\in [0,\infty)}N(v^{1}(t))[M(v^{{1}}(t))]^{\sigma_{c}}=L.
\end{equation}
Next, we claim that
\[
M(v^{1}(t))=1.
\]
Suppose, by contradiction that $M(v^{1}(t))<1$, then
\[
E(v^{1}(t))[M(v^{1}(t))]^{\sigma_{c}}< \delta_{c}.
\]
Hence, by \eqref{LequalV} and definition of $\delta_{c}$ (see \eqref{criticald}), we obtain $\| v^{j}(0, \cdot)\|_{L^{a}_{t}(\R^{+})}<\infty$, which is a contradiction.
Similarly, we can may show that $E(v^{1}(t))= \delta_{c}$. 
Let $u_{c}(t)$ be the solution to equation \eqref{NLS} with initial data $u_{c}(0)=v^{1}(x, 0)$. Then $u_{c}(t)$ is global ($T_{+}=\infty$) and satisfies
\[\begin{split}
M(u_{c}(t))&=M(v^{1}(t))=1, \\
E(u_{c}(t))&=E(v^{1}(t))=\delta_{c},
\end{split}\]
and
\[
\sup_{t\in [0,\infty)}N(u_{c}(t))=\sup_{t\in [0,\infty)}N(v^{1}(t))=L.
\]
Moreover, we have  that $\| u_{c}(0, \cdot)\|_{L^{a}_{t}(\R^{+})}=\infty$.
Finally, we consider the precompactness of $\Omega$. Indeed, notice that for any time sequence $\left\{t_{n}\right\}_{n\geq1}$,
the sequence $u_{c}(t_{n})$ is uniformly bound in $H^{1}(\R)$ and satisfies
\[
M(u_{c}(t_{n}))=1, \quad E(u_{c}(t_{n}))=\delta_{c}\quad \text{and}\quad \| u_{c}(0, \cdot)\|_{L^{a}_{t}((t_{n},\infty))}=\infty.
\]
Hence, regarding $u_{c}(t_{n})$ as the foregoing $\phi_{n}$, and using an argument similar to the above, we can find that
there exists sequences $\left\{\tau^{1}_{n}\right\}_{n\geq1}\subset \R$, $\left\{W^{1}_{n}\right\}_{n\geq1}\subset H^{1}(\R)$ and
$\xi^{1}\in H^{1}(\R)$ such that
\begin{equation}\label{NewUC}
u_{c}(t_{n})=e^{-i\tau^{1}_{n}\partial^{2}_{x}}\xi^{1}+W^{1}_{n},
\end{equation}
with
\[
M(\xi^{1})=1, \quad \lim_{n\to\infty}M(W^{1}_{n})=0 \quad \text{and}\quad \lim_{n\to\infty}E(W^{1}_{n})=0.
\]
In particular, by inequality \eqref{IneEnergy} we infer that
\begin{equation}\label{NormW}
\lim_{n\to\infty}\| W^{1}_{n}  \|_{H^{1}_{x}}\leq \lim_{n\to\infty}\left[\frac{16}{\eta}E(W^{1}_{n})+M(W^{1}_{n})\right]=0.
\end{equation}
Moreover, arguing as Scenario 2, we obtain that $\tau^{1}_{n}\to \tau_{\ast}<\infty$ as $n\to\infty$. Hence,  putting together
\eqref{NewUC}  and \eqref{NormW} we see that $u_{c}(t_{n})$ converges in $H^{1}(\R)$. This proves the proposition.
\end{proof}

\begin{proof}[{Proof of Lemma \ref{ExpanBoundedNLS}}] We follow the ideas of the proof of \cite[Lemma 3.9]{Guevara2012} and \cite[Lemma 3.2]{DinhForceHaja2020}.
First of all notice that, by definition of the intervals $I^{m}_{n}$ (see \eqref{CompactI}), there exists $m_{0}\in \N$ sufficiently large  such that $[0, T]\subset I^{m_{0}}_{n}$  for infinitely many $n$.
We recall that there exists $J_{\ast}\in \N$ such that 
\begin{equation}\label{LimNweJs}
\| v^{j}_{n}(0,\cdot) \|_{L^{a}_{t}(\R^{+})}+\| v^{j}_{n}(0,\cdot) \|_{L^{\infty}_{t}(\R^{+})}\lesssim
\|\psi^{j}\|_{H^{1}}
\end{equation}
for all $j\geq J_{\ast}$. In particular, it follows that for $j\geq J_{\ast}$
\begin{equation}\label{LimitzeroV}
\|  v_{n}^{j}(0,\cdot) \|_{L^{a}_{t}([0,T])}+\|  v_{n}^{j}(0,\cdot) \|_{L^{\infty}_{t}([0,T])} 
\lesssim \|\psi^{j}\|_{H^{1}}.
\end{equation}
Moreover, using {\eqref{DefJi}} we infer that 
\begin{equation}\label{Bounj11}
\limsup_{n\to \infty}\|  v_{n}^{j}(0,\cdot) \|_{L^{a}_{t}([0,T])}<\infty,
\end{equation}
for all $J^{1}_{\ast}\leq j< J_{\ast}$, where $J^{\ast}_{1}$ is given in \eqref{DefJi}. Finally, by \eqref{CompactI} there exists a constant $C=C(T)$ (possibly depending on time $T$) such that
\begin{equation}\label{Jestrebound}
\|  v_{n}^{j}(0,\cdot) \|_{L^{a}_{t}([0,T])}\leq C(T) \quad \text{for $n$ sufficiently large},
\end{equation}
for all $1\leq j< J^{1}_{\ast}$.

Now, by reordering, we can choose $0\leq J_{1}\leq J_{\ast}$ such that
\begin{itemize}[leftmargin=5mm]
		\item $1\leq j\leq J_{1}$: the time shifts $t^{j}_{n}=0$ for every $n$. We write $J_{1}=0$ if there is no such $j$ that satisfies the condition. Notice that \eqref{Ortho} implies that $J_{1}=\left\{0,1\right\}$. 
   \item $J_{1}+1\leq j< J_{\ast}$: the time shifts $\lim_{n\to\infty}|t^{j}_{n}|=\infty$. We write $J_{1}=J_{0}$ if there is no such $j$ that satisfies the condition.
		\end{itemize}
We treat only the case $J_{1}=1$. The case $J_{1}=0$ is easier. Now we set
\[
R:=\max\left\{1, \limsup_{n\to\infty}\sup_{t\in [0,T]}K(u_{n}(t))\right\}<\infty.
\]
Moreover,  $T_{\ast}$ denote the maximal forward time such that the solution $v^{1}(t)$ satisfies
$\sup_{t\in [0,T_{\ast}]}K(v^{1}(t))\leq 2R$. Notice that we may assume that $[0, T_{\ast}]\subset I^{m_{0}}_{n}$ for infinitely
many $n$. 

Using the conservation of mass and \eqref{Normsortho}, 
\[
M( v_{n}^{j}(t))=M( v^{j}(-t^{n}_{j})) \lesssim 1+ M(\psi^{j}).
\]
By inequality above (for $j=1$) and the Gagliardo-Nirenberg inequality \eqref{Ga-Ni-Inequ} we get
\begin{equation}\label{33Pertu}
\| v^{1}(0,\cdot)\|_{L^{\infty}_{t}([0,T_{\ast}])}\leq \sup_{t\in [0,T_{\ast}]}[M( v^{1}(t))K( v^{1}(t))]^{\frac{1}{4}}
\lesssim R^{\frac{1}{4}}.
\end{equation}
Here we recall that $t^{1}_{n}=0$ for all $n$. On the other hand,
if $t^{j}_{n}\to -\infty$, we obtain
\begin{equation*}
\begin{split}
\| v_{n}^{j}(0,\cdot)\|_{L^{\infty}_{t}([0,T_{\ast}])}&\lesssim
\| v_{n}^{j}(t)-e^{i(t-t^{j}_{n})\partial^{2}_{x}}\psi^{j}\|_{L^{\infty}_{t}([0,T_{\ast}], H_{x}^{1}(\R))}\\
&+\|e^{i(t-t^{j}_{n})\partial^{2}_{x}}\psi^{j}\|_{L^{\infty}_{t}([0,T_{\ast}] H_{x}^{1}(\R))}\\
&\lesssim \| v_{n}^{j}(t)-e^{i(t-t^{j}_{n})\partial^{2}_{x}}\psi^{j}\|_{L^{\infty}_{t}([0,T_{\ast}], H_{x}^{1}(\R))}+1
\lesssim 1,
\end{split}
\end{equation*}
for $n$ sufficiently large.  A similar statement is valid for $t^{j}_{n}\to \infty$. Therefore, for each $2\leq j< J_{\ast}$ 
we get the estimate 
\begin{equation}\label{Boundtinfi}
\limsup_{n\to\infty}\| v_{n}^{j}(0,\cdot)\|_{L^{\infty}_{t}([0,T_{\ast}])}<\infty.
\end{equation}

Combining \eqref{Bounj11}, \eqref{Jestrebound}, \eqref{33Pertu} and \eqref{Boundtinfi}  we have that for each $1\leq j< J_{\ast}$
\begin{equation}\label{Ulti44pertu}
\| v^{j}(0,\cdot)\|_{L^{a}_{t}([0,T_{\ast}])\cap L^{\infty}_{t}([0,T_{\ast}])}
\leq C(R, T_{\ast})
\end{equation}
for $n$ sufficiently large

\textit{Step 1.} We show that for all $t\in [0, T_{\ast}]$
\begin{equation}\label{Step1Limit}
N(u_{n}(t))=\sum^{J}_{j=1}N(v_{n}^{j}(t))+N(\tilde{W}^{J}_{n}(t))+o_{n,J}(1) 
\end{equation}
where $o_{J,n}(1)\to$ as $n\to \infty$.  Indeed, we write
\[
{u}^{J}_{n}(x,t)=\sum^{J}_{j=1}v^{j}_{n}(x,t)+\tilde{W}^{J}_{n}(x, t).
\]
Notice that ${u}^{J}_{n}$ satisfies (see \eqref{second-initial})
\[
\begin{cases} 
i\partial_{t}u^{J}+\partial^{2}_{x}u^{J}+\delta(x) |u^{J}|^{p-1}u^{J}=\delta(x) e^{J}_{n}, \quad \text{$x\in \R$, $t\in \R$ } \\
{u}^{J}_{n}(0,x)=\varphi_{n},
\end{cases} 
\]
where
\[
e^{J}_{n}=\left[\sum^{J}_{j=1}F(v^{j}_{n})-F\(\sum^{J}_{j=1}v^{j}_{n}\)\right]
+[(F({u}^{J}_{n})-F(\tilde{W}^{J}_{n}))-F({u}^{J}_{n}-\tilde{W}^{J}_{n})]
\]
with $F(z):=|z|^{p-1}z$. We want to apply Proposition \ref{pertucompact}. We begin by estimating 
$\|{u}^{J}_{n}(0,\cdot) \|_{L_{t}^{a}[0, T_{\ast}]\cap L^{\infty}_{t}[0, T_{\ast}]}$. Indeed,
by Proposition \ref{Estimates-S}, Duhamel formula and \eqref{limnew} we infer that
\begin{equation}\label{AssypWT}
\lim_{J\to\infty}\lim_{n\to\infty}\|\tilde{W}^{J}_{n}(0,\cdot) \|_{L^{a}_{t}([0, T_{\ast}])\cap L^{\infty}_{t}([0, T_{\ast}])}
=0.
\end{equation}
Then, by \eqref{LimitzeroV} and \eqref{Ulti44pertu}  we obtain that there exists a constant 
$C>0$ such that for $J< J_{\ast}$,
\[
 \limsup_{n\rightarrow\infty}\|{u}^{J}_{n}(0,\cdot) \|_{L_{t}^{a}[0, T_{\ast}]\cap L^{\infty}_{t}[0, T_{\ast}]}\lesssim 1.
\]
Moreover, for $J\geq J_{\ast}$, and using \eqref{LimitzeroV} and \eqref{Normsortho} we infer that
\begin{equation}\label{uniforinJ}
\limsup_{n\rightarrow\infty} \|{u}^{J}_{n}(0,\cdot) \|_{L_{t}^{a}[0, T_{\ast}]\cap L^{\infty}_{t}[0, T_{\ast}]}\lesssim 1,
\end{equation}
uniformly in $J$. On the other hand, we claim that
\begin{equation}\label{Claimerror}
\lim_{J\to\infty}\lim_{n\to\infty}\| e^{J}_{n}(0, \cdot) \|_{L_{t}^{b}[0, T_{\ast}]\cap L^{\infty}_{x}[0, T_{\ast}]}=0.
\end{equation}
Indeed, by \eqref{AssypWT} and \eqref{uniforinJ}  we have
\[\begin{split}
&\lim_{J\to\infty}\lim_{n\to\infty}\| [(F({u}^{J}_{n})-F(\tilde{W}^{J}_{n}))-F({u}^{J}_{n}-\tilde{W}^{J}_{n})](0, \cdot)  \|_{L_{t}^{b}[0, T_{\ast}]\cap L^{\infty}_{t}[0, T_{\ast}]}\\
&\lesssim \lim_{J\to\infty}\lim_{n\to\infty}\| [|{u}^{J}_{n}|^{p-1}\tilde{W}^{J}_{n}+|\tilde{W}^{J}_{n}|^{p}](0, \cdot)  \|_{L_{t}^{b}[0, T_{\ast}]\cap L^{\infty}_{t}[0, T_{\ast}]}\\
&\lesssim \lim_{J\to\infty}\lim_{n\to\infty}[\|\tilde{W}^{J}_{n}(0,\cdot) \|_{L^{a}_{t}([0, T_{\ast}])\cap L^{\infty}_{t}([0, T_{\ast}])}
+\|\tilde{W}^{J}_{n}(0,\cdot) \|^{p}_{L^{a}_{t}([0, T_{\ast}])\cap L^{\infty}_{t}([0, T_{\ast}])}]\\
&=0.
\end{split}\]
Here we have used that 
\[\begin{split}
\| [|{u}^{J}_{n}|^{p-1}\tilde{W}^{J}_{n}](0, \cdot) \|_{L_{t}^{b}[0, T_{\ast}]\cap L^{\infty}_{t}([0, T_{\ast}])}\\
\leq \|{u}^{J}_{n}(0, \cdot)\|^{p-1}_{L_{t}^{a}[0, T_{\ast}]\cap L^{\infty}_{t}([0, T_{\ast}])}
\|\tilde{W}^{J}_{n}(0,\cdot) \|_{L^{a}_{t}([0, T_{\ast}])\cap L^{\infty}_{t}([0, T_{\ast}])}.
\end{split}
\]
In addition, using the inequality
\[\begin{split}
&\left\|\left[\sum^{J}_{j=1}F(v^{j}_{n})-F\(\sum^{J}_{j=1}v^{j}_{n}\)\right](0, \cdot)\right\|
_{L^{b}_{t}([0, T_{\ast}])\cap L^{\infty}_{t}([0, T_{\ast}])}\\
&\lesssim \sum_{j\neq k}\||v^{j}_{n}|^{p-1}|v^{k}_{n}|\|_{L^{b}_{t}([0, T_{\ast}])\cap L^{\infty}_{t}([0, T_{\ast}])}
\end{split}
\]
together with the orthogonality \eqref{Ortho}, one can easily show (see proof of Theorem 1.4 in \cite{AdamiFukuHolmer}) that 
\[
\lim_{J\to\infty}\lim_{n\to\infty}\left\|\left[\sum^{J}_{j=1}F(v^{j}_{n})-F\(\sum^{J}_{j=1}v^{j}_{n}\)\right](0, \cdot)\right\|
_{L^{b}_{t}([0, T_{\ast}])\cap L^{\infty}_{t}([0, T_{\ast}])}=0.
\]
This proves the claim \eqref{Claimerror}. Thus, Lemma \ref{Perturbationlemma} implies 
\begin{equation}\label{LimitUUJ}
\lim_{J\rightarrow \infty}\lim_{n\rightarrow \infty}
\| u_{n}(0,\cdot)-u_{n}^{J}(0,\cdot)\|_{L^{a}_{t}([0,T_{\ast}])\cap L^{\infty}_{t}([0,T_{\ast}])}=0.
\end{equation}
Then,  from \eqref{LimitUUJ}, we can repeat the same argument developed in \cite[Proposition 3.1]{AdamiFukuHolmer} to obtain
\begin{equation}\label{LIMITinN}
N(u_{n}(t))=\sum^{J}_{j=1}N(v_{n}^{j}(t))+N(\tilde{W}^{J}_{n}(t))+o_{n,J}(1), \quad \text{for all $t\in [0,T_{\ast}]$}.
\end{equation}

\textit{Step 2.} Conclusion. First, notice that putting together \eqref{PythagoreanEx} and \eqref{Limone}, by the conservation 
the energy we get
\begin{equation}\label{EnerH}
\begin{split}
E(u_{n}(t))&=\sum^{J}_{j=1}E(v_{n}^{j}(t))+E(\tilde{W}^{J}_{n}(t))+o_{n,J}(1) \quad \text{for all $t\in [0,T_{\ast}]$}.
\end{split}
\end{equation}
Then, combining \eqref{LIMITinN} and \eqref{EnerH} we get \eqref{DecomK}.
Finally, we also note that $T\leq T_{\ast}$. Indeed, by contradiction assume that
$T_{\ast}<T$. Since $t^{1}_{n}=0$ for all $n$, by \eqref{DecomK} it follows that 
\[
\sup_{t\in [0,T_{\ast}]}K(v^{1}(t))
\leq \sup_{t\in [0,T_{\ast}]}K(u_{n}(t))\leq \sup_{t\in [0,T]}K(u_{n}(t))\leq R,
\] 
 which is a contradiction with the choice of $T_{\ast}$. This proves that $T\leq T_{\ast}$.
\end{proof}

\begin{proof}[{Proof of Theorem \ref{Th1}}]
Assume that $\delta_{c}<\infty$, then the Proposition \ref{Compactness-critical} implies that there exists a critical element 
$u_{c,0}\neq 0$ such that the corresponding solution to equation \eqref{NLS} verifies that the set 
$\left\{u(t): t\geq 0\right\}\subset H^{1}(\R)$ is precompact in $H^{1}(\R)$. Now, by inequality \eqref{IneG} 
and using the same argument as in \cite[Theorem 6.1]{HolmerRoudenko2008} we obtain $u_{c,0}= 0$, which is a contradiction. This proves that $\delta_{c}=+\infty$,
which implies  Theorem \ref{Th1}.
\end{proof}


\section{Scattering results below and above the ground states threshold}\label{S:2} 

In this section we show the Theorems \ref{ScattebelowQ} and \ref{ScatteaboveQ}.

\begin{proof}[{Proof of Theorem \ref{ScattebelowQ}}]
We proceed in two steps.\\
\textit{Step 1.} 
We show that there exists $\theta=\theta(Q, u_{0})>0$ small such that
\begin{equation}\label{kinecMass-ine}
K(u(t))[M(u(t))]^{\sigma_{c}}\leq (1-\theta)K(Q)[M(Q)]^{\sigma_{c}} \quad \text{for all $t\in [0, T_{+})$.}
\end{equation}
Indeed, from the Gagliardo-Nirenberg inequality \eqref{Ga-Ni-Inequ}  we obtain
\begin{align}\nonumber
	E(u(t))[M(u(t))]^{\sigma_{c}}&=\frac{1}{2}K(u(t))[M(u(t))]^{\sigma_{c}}-\frac{1}{p+1}N(u(t))[M(u(t))]^{\sigma_{c}}\\\nonumber
	&\geq \frac{1}{2}K(u(t))[M(u(t))]^{\sigma_{c}}-\frac{1}{p+1}[K(u(t))]^{\frac{p+1}{4}}[M(u(t))]^{\sigma_{c}+\frac{p+1}{4}}\\
	&=\frac{1}{2}K(u(t))[M(u(t))]^{\sigma_{c}}-\frac{1}{p+1}(K(u(t))[M(u(t))]^{\sigma_{c}})^{\frac{p+1}{4}} \nonumber
	\\\label{InterEM}
	&=\Phi(K(u(t))[M(u(t))]^{\sigma_{c}}),
\end{align}
where
\[
\Phi(x):=\frac{1}{2}x-\frac{1}{p+1} x^{\frac{p+1}{4}}.
\]
Combining the identities \eqref{PohoIdenti} and \eqref{OptConst} we see that
\[\begin{split}
\Phi(K(Q)[M(Q)]^{\sigma_{c}})&=\frac{1}{2}K(Q)[M(Q)]^{\sigma_{c}}-\frac{2}{p+1}N(Q)[M(Q)]^{\sigma_{c}}\\
&=\frac{(p-3)}{2(p+1)}K(Q)[M(Q)]^{\sigma_{c}}=E(Q)[M(Q)]^{\sigma_{c}}.
\end{split}\]
Then, by assumption \eqref{EMass} and conservation laws,
\[
\Phi(K(u(t))[M(u(t))]^{\sigma_{c}})\leq E(u_{0})[M(u_{0})]^{\sigma_{c}}
<E(Q)[M(Q)]^{\sigma_{c}}=\Phi(K(Q)[M(Q)]^{\sigma_{c}}),
\]
for every $t\in [0, T_{+})$. Thus, by continuity and assumption \eqref{KinectMass}  we infer that 
\begin{equation}\label{KMm1}
K(u(t))[M(u(t))]^{\sigma_{c}}<K(Q)[M(Q)]^{\sigma_{c}}\quad \text{for all $t\in [0, T_{+})$.}
\end{equation}
Again, by assumption \eqref{EMass} we have that there exists $a>0$ small such that
\begin{equation}\label{AssuptiNe}
E(u_{0})[M(u_{0})]^{\sigma_{c}}<(1-a)E(Q)[M(Q)]^{\sigma_{c}}.
\end{equation}
In addition, from identities \eqref{EnergyIdentities} and \eqref{OptConst} we get
\[
E(Q)[M(Q)]^{\sigma_{c}}=\frac{(p-3)}{2(p+1)}K(Q)[M(Q)]^{\sigma_{c}}=\frac{(p-3)}{4(p+1)}(K(Q)[M(Q)]^{\sigma_{c}})^{\frac{(p+1)}{4}}.
\]
Thus, by equation above and inequality \eqref{InterEM} we obtain
\[
\frac{(p+1)}{(p-3)}\(\frac{K(u(t))[M(u(t))]^{\sigma_{c}}}{K(Q)[M(Q)]^{\sigma_{c}}}\)-
\frac{4}{p-3}\(\frac{K(u(t))[M(u(t))]^{\sigma_{c}}}{K(Q)[M(Q)]^{\sigma_{c}}}\)^{\frac{p+1}{4}}\leq 1-a.
\]
Now we set
\[
\Psi(x)=\frac{(p+1)}{(p-3)}x-\frac{4}{p-3}x^{\frac{p+1}{4}} \quad \text{ for all $x\in (0,1)$}.
\]
Notice that $\Psi(0)=0$, $\Psi(1)=1$ and $\Psi^{\prime}(x)>0$ for all $x\in (0,1)$. Since
\[
\Psi\(\frac{K(u(t))[M(u(t))]^{\sigma_{c}}}{K(Q)[M(Q)]^{\sigma_{c}}}\)\leq 1-a,
\]
it follows from \eqref{KMm1} that there exists $\theta=\theta(a)>0$ such that $x<1-\theta$, which implies \eqref{kinecMass-ine}.\\

\textit{Step 2.} Conclusion. By Gagliardo-Nirenberg inequality \eqref{Ga-Ni-Inequ} and \eqref{kinecMass-ine} we deduce
\begin{align*}
	N(u(t))[M(u(t))]^{\sigma_{c}}&\leq 	[K(u(t))]^{\frac{p+1}{4}}[M(u(t))]^{\frac{p+1}{4}}[M(u(t))]^{\sigma_{c}}\\
		&= (K(u(t))[M(u(t))]^{\sigma_{c}})^{\frac{p+1}{4}}\\
		&<(1-\theta)^{\frac{p+1}{4}}(K(Q)[M(Q)]^{\sigma_{c}})^{\frac{p+1}{4}} \quad \text{for all $t\in [0, T_{+})$.}
\end{align*}
 Thus, since
\[
(K(Q)[M(Q)]^{\sigma_{c}})^{\frac{p+1}{4}}=(K(Q)[M(Q)])^{\frac{p+1}{4}}[M(Q)]^{\sigma_{c}}=N(Q)[M(Q)]^{\sigma_{c}},
\] 
we find that
\[
N(u(t))[M(u(t))]^{\sigma_{c}}<(1-\theta)^{\frac{p+1}{4}}N(Q)[M(Q)]^{\sigma_{c}}\quad \text{for all $t\in [0, T_{+})$,}
\]
which implies \eqref{InequSc}. So the scattering criterion, Theorem \ref{Th1}, tell us that $u(t)$ scatters.
This completes the proof of Theorem \ref{ScattebelowQ}.
\end{proof}

Before establishing the following result, we recall that the virial quantity satisfies the following identities
\[
V^{\prime}(t)=4\IM\int_{\R}x\partial_{x}u(x,t)\overline{u}(x,t)dx
\]
and
\begin{align}\label{Viria11}
	V^{\prime\prime}(t)&=8(K(u(t))-\frac{1}{2}N(u(t))\\\label{Viria22}
	&=4(p+1)E(u(t))-2(p-3)K(u(t))\\\label{Viria33}
	&=16E(u(t))+4\frac{(3-p)}{p+1}N(u(t)).
	\end{align}

We use the following Cauchy-Schwartz inequality.

\begin{lemma}\label{CauchyS}
Let  $f\in H^{1}(\R)$ such that $|x|f\in L^{2}(\R)$. Then
\begin{equation}\label{InequBanica}
\(\IM \int_{\R}x\partial_{x}f\overline{f} dx\)^{2}
\leq  \int_{\R}|x|^{2}|f|^{2}dx \left[K(f)-\frac{[N(f)]^{\frac{4}{p+1}}}{M(f)}\right].
\end{equation}
\end{lemma}
\begin{proof}
We follow a similar argument as in \cite[Lemma 2.1]{DR5}. 
Given  $\lambda\in \R$,  a simple calculation shows that
\[   
K(e^{i\lambda |x|^{2}}f)=4\lambda^{2}\int_{\R}|x|^{2}|f|^{2}dx 
+4\lambda \IM \int_{\R}x\partial_{x}f\overline{f}dx
+K(f).
\]
Moreover, $M(e^{i\lambda |x|^{2}}f)=M(f)$ and $N(e^{i\lambda |x|^{2}}f)=N(f)$. Now we define the  quadratic polynomial in $\lambda$,
\begin{align*}
\Phi(\lambda)&=K(e^{i\lambda |x|^{2}}f)-\frac{[N(e^{i\lambda |x|^{2}}f)]^{\frac{4}{p+1}}}{M(e^{i\lambda |x|^{2}}f)}\\
&=4\lambda^{2}\int_{\R}|x|^{2}|f|^{2}dx
+4\lambda \IM \int_{\R}x\partial_{x}f\overline{f}dx
+\(K(f)-\frac{[N(f)]^{\frac{4}{p+1}}}{M(f)}\)
\end{align*}
By using the Gagliardo-Nirenberg inequality \eqref{Ga-Ni-Inequ} we infer that $\Phi(\lambda)\geq 0$, this implies that
the discriminant of $\Phi$ is non-positive, which implies the inequality \eqref{InequBanica}.
\end{proof}

We are now ready to give the proof of Theorem \ref{ScatteaboveQ}.

\begin{proof}[Proof of Theorem \ref{ScatteaboveQ}] We adapt here a proof given in \cite[Theorem 1.4]{DR5}.
From \eqref{Viria22} and  \eqref{Viria33} we have that
\[
K(u(t))=\frac{1}{2(p-3)}(4(p+1)E(u(t))- V^{\prime\prime}(t)), \quad N(u(t))=\frac{(p+1)}{4(p-3)}(16E(u(t))- V^{\prime\prime}(t)).
\]
Notice that as  $N(u(t))\geq 0$ it follows  $V^{''}(t)\leq 16E(u(t))$. Moreover, by inequality \eqref{InequBanica} and identities above
we get 
\[\begin{split}
(V^{'}(t))^{2}\leq 16V(t)\(\frac{1}{2(p-3)}(4(p+1)E(u(t))- V^{\prime\prime}(t))
\right.
\\
\left.
-[M(u(t))]^{-1}\left[
\frac{(p+1)}{4(p-3)}(16E(u(t))- V^{\prime\prime}(t))
\right]^{\frac{4}{p+1}}
\).
\end{split}\]
We set 
\begin{equation}\label{Phifun}
\Phi(x)=\frac{1}{2(p-3)}(4(p+1)E(u_{0})-x)-[M(u_{0})]^{-1}\(
\frac{(p+1)}{4(p-3)}(16E(u_{0})- x)
\)^{\frac{4}{p+1}}
\end{equation}
for all $x\leq 16 E(u_{0})$. Thus we have that
\begin{equation}\label{InVirial}
(z^{\prime}(t))^{2}\leq 4\Phi(V^{\prime\prime}(t)),\quad \mbox{where} \quad z(t):=\sqrt{V(t)}.
\end{equation}
Now notice that
\[
\Phi^{\prime}(x)=-\frac{1}{2(p-3)}+\frac{[M(u_{0})]^{-1}}{(p-3)}\(
\frac{(p+1)}{4(p-3)}(16E(u_{0})- x)\)^{-\frac{(p-3)}{p+1}}.
\]
Since $p>3$, we deduce that exists a unique point $x_{0}$ such that $\Phi^{\prime}(x_{0})=0$, where $x_{0}$ satisfies
\[
\frac{1}{2(p-3)}=\frac{[M(u_{0})]^{-1}}{(p-3)}\(
\frac{(p+1)}{4(p-3)}(16E(u_{0})- x_{0})\)^{-\frac{(p-3)}{p+1}}.
\]

At the same time, as $p>3$, we have that $\Phi(x)$ is decreasing on $(-\infty, x_{0})$
and increasing on the interval $(x_{0}, 16 E(u_{0}))$. Moreover, by equation above we infer that
\begin{equation}\label{Xizero}
[M(u_{0})]^{-1}=\frac{1}{2}\(
\frac{(p+1)}{4(p-3)}(16E(u_{0})- x_{0})\)^{\frac{(p-3)}{p+1}},
\end{equation}
which implies
\[\begin{split}
\Phi(x_{0})&=\frac{1}{2(p-3)}(4(p+1)E(u_{0})-x_{0})\\
&-\frac{1}{2}\(
\frac{(p+1)}{4(p-3)}(16E(u_{0})- x_{0})\)^{\frac{(p-3)}{p+1}}\(
\frac{(p+1)}{4(p-3)}(16E(u_{0})- x_{0})
\)^{\frac{4}{p+1}}\\
&=\frac{x_{0}}{8}.
\end{split}\]
By \eqref{Xizero} we obtain
\begin{equation}\label{primer-ident}
\(\frac{p+1}{4(p-3)}\)[M(u_{0})]^{\sigma_{c}}(16E(u_{0})-x_{0})=2^{\sigma_{c}}.
\end{equation}
Similarly, from identities \eqref{EnergyIdentities} and \eqref{OptConst} we get
\begin{equation}\label{second-ident}
2\(\frac{2(p+1)}{p-3}\)E(Q)[M(Q)]^{\sigma_{c}}=2^{\sigma_{c}}.
\end{equation}
Therefore, combining \eqref{primer-ident} and \eqref{second-ident} we deduce 
\[
\frac{16E(Q)[M(Q)]^{\sigma_{c}}}{[M(u_{0})]^{\sigma_{c}}(16E(u_{0})-x_{0})}=1,
\]
equivalently,
\begin{equation}\label{PosiIdent}
\frac{E(u_{0})M(u_{0})]^{\sigma_{c}}}{E(Q)[M(Q)]^{\sigma_{c}}}\(1-\frac{x_{0}}{16E(u_{0})}\)=1.
\end{equation}
In particular,
\begin{equation}\label{Positivexzero}
x_{0}=16E(u_{0})\(1-\frac{E(Q)M(Q)]^{\sigma_{c}}}{E(u_{0})[M(u_{0})]^{\sigma_{c}}}\).
\end{equation}
Notice that by assumption \eqref{assum11}, we infer that $x_{0}\geq 0$.
On the other hand, in view of assumption \eqref{assum33}, \eqref{EnergyIdentities} and \eqref{Viria33} we get 
\begin{equation}\label{Vmxero}
\begin{split}
	V^{\prime\prime}(0)&=16E(u_{0})-\frac{4(3-p)}{(p+1)}N(u_{0})\\
	&>16E(u_{0})-\frac{4(p-3)}{(p+1)}\frac{N(Q)[M(Q)]^{\sigma_{c}}}{[M(u_{0})]^{\sigma_{c}}}\\
	&=16E(u_{0})\(1- \frac{E(Q)[M(Q)]^{\sigma_{c}}}{E(u_{0})[M(u_{0})]^{\sigma_{c}}}\)\\
	&=x_{0}.
\end{split}
\end{equation}
Furthermore, the assumption \eqref{assump44} means
\begin{equation}\label{positivederivate}
z^{\prime}(0)\geq 0
\end{equation}
and assumption \eqref{assump22} implies that
\begin{equation}\label{boundz}
(z^{\prime}(0))^{2}\geq\frac{x_{0}}{2}=4\Phi(x_{0}).
\end{equation}

We prove the theorem in two steps as follows.

\textit{Step 1.} We show that there exists $\delta_{0}>0$ such that 
\begin{equation}\label{EstimativeV}
V^{\prime\prime}(t)\geq x_{0}+\delta_{0},
\quad \text{for all $t\in [0, T_{+})$}.
\end{equation}
Indeed, by continuity and \eqref{Vmxero}, we infer that there exists $\delta_{1}>0$ and $t_{0}>0$ such that
\begin{equation}\label{Vsmall}
V^{\prime\prime}(t)>x_{0}+\delta_{1} \quad \text{for all $t\in [0, t_{0})$}.
\end{equation}
On the other hand, taking $t_{0}$ more smaller if necessary, we may assume that
\begin{equation}\label{Zprimeposi}
z^{\prime}(t_{0})>2\sqrt{\Phi(x_{0})}.
\end{equation}
Indeed, if $z^{\prime}(0)>2\sqrt{\Phi(x_{0})}$, then by continuity we have the result. If $z^{\prime}(0)=2\sqrt{\Phi(x_{0})}$,
it follows from \eqref{Vmxero} that 
\[
z^{\prime\prime}(0)=\frac{1}{z(0)}\(\frac{V^{\prime\prime}(0)}{2}-(z(0))^{2} \)
>\frac{1}{z(0)}\(\frac{x_{0}}{2}- \frac{x_{0}}{2}\)
=0,
\]
hence, \eqref{Zprimeposi} holds for some $t_{0}>0$. 
We define $\epsilon_{0}>0$ such that
\begin{equation}\label{Zepsy}
z^{\prime}(t_{0})\geq 2\sqrt{\Phi(x_{0})}+2\epsilon_{0}.
\end{equation}
We claim that
\begin{equation}\label{Zineqfor}
z^{\prime}(t)> 2\sqrt{\Phi(x_{0})}+\epsilon_{0} \quad \text{for all $t\geq t_{0}$}.
\end{equation}
Indeed, suppose that the claim \eqref{Zineqfor} is false and let
\[
t_{1}:=\inf\left\{t\geq t_{0}; z^{\prime}(t)\leq  2\sqrt{\Phi(x_{0})}+\epsilon_{0}  \right\}.
\]
By using the  continuity of $z^{\prime}(t)$, we get
\begin{equation}\label{zcontra}
z^{\prime}(t)\geq 2\sqrt{\Phi(x_{0})}+\epsilon_{0}, \quad \text{for all $t\in [t_{0}, t_{1}]$}
\end{equation}
with 
\begin{equation}\label{znopoint}
z^{\prime}(t_{1})=2\sqrt{\Phi(x_{0})}+\epsilon_{0}.
\end{equation}
By inequality \eqref{InVirial} and \eqref{zcontra} we see that
\begin{equation}\label{Doubleinequ}
(2\sqrt{\Phi(x_{0})}+\epsilon_{0})^{2}\leq (z^{\prime}(t))^{2}\leq 4\Phi(V^{\prime\prime}(t))\quad 
\text{for all $t\in [t_{0}, t_{1}]$}.
\end{equation}
Thus, we deduce $\Phi(V^{\prime\prime}(t))>\Phi(x_{0})$ for all $t\in [t_{0}, t_{1}]$, that is,  $V^{\prime\prime}(t)\neq x_{0}$
and by continuity $V^{\prime\prime}(t)>x_{0}$ for all $t\in [t_{0}, t_{1}]$.
Next we show that there exists a universal positive constant $L$ such that
\begin{equation}\label{universalconstan}
V^{\prime\prime}(t)\geq  x_{0}+\frac{\sqrt{\epsilon_{0}}}{L} \quad
\text{for all $t\in [t_{0}, t_{1}]$}.
\end{equation}
We consider two cases:\\
(i) If $V^{\prime\prime}(t)\geq x_{0}+1$, then for $L>0$ sufficiently large \eqref{universalconstan} holds.\\
(ii) Assume that  $x_{0}<V^{\prime\prime}(t)<x_{0}+1$. By the Taylor expansion of $\Phi$ around of $x_{0}$, we infer that
there exists $r>0$ such that
\[
\Phi(x)\leq \Phi(x_{0})+r(x-x_{0})^{2} \quad \text{when $|x-x_{0}|\leq 1$},
\]
which implies, by \eqref{Doubleinequ},
\[
4\sqrt{\Phi(x_{0})}\epsilon_{0}<4 a(V^{\prime\prime}(t)-x_{0})^{2},
\]
and taking $L=\sqrt{r}(\Phi(x_{0}))^{-\frac{1}{4}}$, \eqref{universalconstan} holds. 

Now, we may use \eqref{Zineqfor} and \eqref{universalconstan} to obtain
\begin{align*}
	z^{\prime\prime}(t_{1})&=\frac{1}{z(t_{1})}\(\frac{V^{\prime\prime}(t)}{2}-(z(t_{1}))^{2} \)\\
	                 &\geq\frac{1}{z(t_{1})}\(\frac{x_{0}}{2}+ \frac{\sqrt{\epsilon_{0}}}{L}- (2\sqrt{\Phi(x_{0})}+\epsilon_{0})^{2} \)\\
									&\geq\frac{1}{z(t_{1})}\(\frac{\sqrt{\epsilon_{0}}}{L}-4\epsilon_{0}\sqrt{\Phi(x_{0})}-\epsilon^{2}_{0}\)>0
\end{align*}
when $\epsilon_{0}$ is small enough, which is a contradiction with \eqref{znopoint}. Thus we obtain the claim \eqref{Zineqfor}.  
Consequently, we have 
\begin{equation}\label{FinalInequ}
V^{\prime\prime}(t)\geq  x_{0}+\frac{\sqrt{\epsilon_{0}}}{L}, \quad
\text{for all $t\in [t_{0}, T_{+})$}.
\end{equation}
Hence, using \eqref{Vsmall} and \eqref{FinalInequ} we obtain \eqref{EstimativeV} with $\delta_{0}=\min\left\{\frac{\sqrt{\epsilon_{0}}}{L}, \delta_{1} \right\}$.

\textit{Step 2.} Conclusion. 
Combining \eqref{EstimativeV}, \eqref{PosiIdent} and \eqref{EnergyIdentities} we have 
\begin{align*}
N(u(t))	[M(u(t))]^{\sigma_{c}}&=\frac{(p+1)}{4(p-3)}(16E(u_{0})-V^{\prime\prime}(t))[M(u_{0})]^{\sigma_{c}}\\
&\leq \frac{(p+1)}{4(p-3)}(16E(u_{0})-x_{0}-\delta_{0})[M(u_{0})]^{\sigma_{c}}\\
&= \frac{4(p+1)}{(p-3)}E(Q)	[M(Q)]^{\sigma_{c}}-\delta_{0} \frac{(p+1)}{4(p-3)}[M(u_{0})]^{\sigma_{c}}\\
&=N(Q)[M(Q)]^{\sigma_{c}}-\delta_{0} \frac{(p+1)}{4(p-3)}[M(u_{0})]^{\sigma_{c}}\\
&=(1-\theta)N(Q)[M(Q)]^{\sigma_{c}} \quad \text{for all $t\in [0, T_{+})$},
\end{align*}
with $\theta=\delta_{0}\frac{(p+1)[M(u_{0})]^{\sigma_{c}}}{4(p-3)N(Q)[M(Q)]^{\sigma_{c}}}$. Thus, by the scattering criterion in Theorem \ref{Th1}, we infer that the solution $u(t)$ scatters in $H^{1}(\R)$ forward in time.
\end{proof}

\section{Long time dynamics at threshold}

 In this section we show the Theorem \ref{DynaThre}. We start with the following result.

\begin{lemma}\label{Compact}
If $\left\{u_{n}\right\}_{n\geq 1}$ is a sequence of $H^{1}(\R)$ such that
\[
E(u_{n})=E(Q),\quad M(u_{n})=M(Q)\quad \text{and}\quad 
 \lim_{n\rightarrow \infty}K(u_{n})=K(Q),
\]
then there exists $\theta\in \R$ such that, possibly for a subsequence only,
\[
u_{n}\rightarrow e^{i\theta}Q \quad \text{strongly in $H^{1}(\R)$ as $n\to \infty$.}
\]
\end{lemma}
\begin{proof}
Let $\left\{u_{n}\right\}_{n\geq 1}\subset H^{1}(\R)$ be a sequence such that $E(u_{n})=E(Q)$, $M(u_{n})=M(Q)$ and $\lim_{n\rightarrow \infty}K(u_{n})=K(Q)$. We observe that
\begin{equation}\label{CoverN}
\lim_{n\rightarrow \infty}N(u_{n})=N(Q).
\end{equation}
Notice also that the sequence $\left\{u_{n}\right\}_{n\geq 1}$ is bounded in $H^{1}(\R)$. Thus, as $H^{1}(\R)$ is reflexive, we infer that 
there exists $v\in H^{1}(\R)$ such that, possibly for a subsequence only, $u_{n}\rightharpoonup v$ weakly in $H^{1}(\R)$ and 
$u_{n}(x)\to v(x)$ a.e. $x\in \R$. Moreover, since $H^{1}[-1,1]\hookrightarrow C[-1,1]$
is compact, it follows that 
\begin{equation}\label{Pointconvergence}
u_{n}(0)\to v(0) \quad \text{ as $n\rightarrow\infty$.}
\end{equation}
Putting together \eqref{CoverN} and \eqref{Pointconvergence} we get
\begin{equation}\label{NQNv}
N(v)=\lim_{n\rightarrow \infty}N(u_{n})=N(Q).
\end{equation}
Thus, from the weak convergent, \eqref{NQNv} and the Gagliardo-Nirenberg inequality \eqref{Ga-Ni-Inequ}  we have
\begin{equation}\label{Unique}
\begin{split}
	N^{\frac{4}{p+1}}(Q)&=N^{\frac{4}{p+1}}(v)\\
	&\leq M(v)K(v)\\
	&\leq [\lim_{n\rightarrow \infty}M(u_{n})][\lim_{n\rightarrow \infty}K(u_{n})]\\
	&\leq M(Q)K(Q).
\end{split}
\end{equation}
Therefore, $N^{\frac{4}{p+1}}(v)=M(v)K(v)$, which implies that $v(x)=e^{i\theta}\lambda Q(rx)$ 
for some $\theta$, $\lambda$, $r\in \R$ (see \cite[Proposition 1.3]{HOLMER2020123522}). 
Since $N(Q)=N(v)$ and $M(Q)=M(v)$ we see that 
$\lambda=1$ and $r=1$. Finally, by using the fact that $M(v)=M(Q)=M(u_{n})$ and $K(v)=K(Q)=\lim_{n\rightarrow\infty}K(u_{n})$,
we infer that $u_{n}\rightarrow v=e^{i\theta} Q$ strongly in $H^{1}(\R)$ as $n\to \infty$. This completes the proof of lemma.
\end{proof}

\begin{lemma}\label{Blow-up-Lemma}
Let $u_{0}\in H^{1}(\R)$. Let $u(t)$ be the solution to Cauchy problem \eqref{NLS} defined
on the maximal forward time $[0, T_{+})$. If
\begin{equation}\label{GGblowup}
\sup_{t\in [0, T_{+})}G(u(t))\leq -\eta,
\end{equation}
for some $\eta>0$, then the solution $u(t)$ blows up in finite time, i.e., $T_{+}<\infty$.
\end{lemma}
\begin{proof}
The proof is essentially given in \cite{HOLMER2020123522}. For the convenience of the reader, we give the details.
Consider the local virial identity,
\[
I(t)=\int_{\R}a(x)|u(x,t)|^{2}dx,
\]
then a calculation leads to
\[
I^{\prime\prime}(t)=4\int_{\R}a^{\prime\prime}(x)|\partial_{x}u(x,t)|^{2}dx-
2a^{\prime\prime}(0)N(u(t))-\int_{\R}a^{\prime\prime\prime\prime}(x)|u(x,t)|^{2}dx,
\]
for any $a\in C^{4}(\R)$ with $a(0)=a^{\prime}(0)=a^{\prime\prime\prime}(0)=0$. Moreover, 
we can choose $a\in C^{4}(\R)$ such that $0\leq a(x)\leq C\epsilon^{-2}$, $a^{\prime\prime}(x)\leq 2$, $a^{\prime\prime}(0)= 2$ and 
$|a^{\prime\prime\prime\prime}(x)|\leq C\epsilon^{2}$ for all $x\in \R$; see proof of Theorem 1.4 in \cite{HOLMER2020123522}.
Thus, we obtain that
\[
I^{\prime\prime}(t)\leq 8K(u(t))-4N(u(t))+C\epsilon^{2}M(u_{0})=G(u(t))+C\epsilon^{2}M(u_{0}),
\]
for $t\in [0, T_{+})$. Therefore, the hypothesis \eqref{GGblowup} implies that
\[
I^{\prime\prime}(t)\leq -\eta+C\epsilon^{2}M(u_{0})\quad \text{for all $t\in [0, T_{+})$}.
\]
Suppose that $T_{+}=\infty$. Choosing $\epsilon$ sufficiently small, we find $t_{\ast}>0$ such that $I(t_{\ast})<0$,
which is a contradiction because $I(t)\geq 0$. Thus, we conclude the proof of blow-up.
\end{proof}

We are now ready to give the proof of Theorem \ref{DynaThre}.

\begin{proof}[{Proof of Theorem \ref{DynaThre}}]
First of all note that since \eqref{EMassThre} is scale-invariant under the scaling
\begin{equation}\label{Scaling}
u^{\lambda}_{0}(x)=\lambda^{\frac{1}{p-1}}u_{0}(\lambda x) \quad \lambda>0,
\end{equation}
taking $\lambda^{\frac{(p-3)}{(p-1)}}=\frac{M(Q)}{M(u_{0})}$, we may assume that
\begin{equation}\label{equalmass-ener}
M(u_{0})=M(Q), \quad E(u_{0})=E(Q).
\end{equation}

(i) If $u_{0}$ satisfies \eqref{DynaKinectMass}, then again by the scaling above we have 
\begin{equation}\label{Kinecscaling}
K(u_{0})<K(Q).
\end{equation}
We claim that 
\begin{equation}\label{ClaimKinecs}
K(u(t))<K(Q) \quad \text{for all $t\in [0, T_{+})$},
\end{equation}
where $u(t)$ is the corresponding solution to Cauchy problem \eqref{NLS} with initial data $u_{0}$ defined on 
the maximal forward time lifespan $[0, T_{+})$.  Indeed, suppose by contradiction that there exists $t^{\ast}\in [0, T_{+})$
such that $K(u(t^{\ast}))=K(Q)$. By \eqref{equalmass-ener} we infer that
\[
K(u(t^{\ast}))=K(Q), \quad M(u(t^{\ast}))=M(Q), \quad N(u(t^{\ast}))=N(Q).
\]
By the same argument as in proof of Lemma \ref{Compact} (see \eqref{Unique}) we obtain that $u(t^{\ast})=e^{i\theta}Q$ for some $\theta\in \R$. Thus, by the uniqueness
of the solution for the Cauchy problem \eqref{NLS} we see that $u(x,t)=e^{i\theta}e^{i(t-t^{\ast})}Q(x)$, which is a 
contradiction with \eqref{Kinecscaling}. This implies that \eqref{ClaimKinecs} holds. In particular, the solution $u(t)$ is global, i.e., 
$T_{+}=\infty$.

Next we consider two cases:\\
\textit{Case 1.} Suppose
\[
\sup_{t\in [0, \infty)}K(u(t))<K(Q).
\]
In this case, we have that there exists $a>0$ small such that $K(u(t))<(1-a)K(Q)$ for all $t\geq 0$. Then by the 
same argument as in proof of Theorem \ref{ScattebelowQ}-\textit{Step 2}, we can show that there exists $b>0$ small such that
\[
N(u(t))[M(u(t))]^{\sigma_{c}}< (1-b)N(Q)[M(Q)]^{\sigma_{c}}\quad \text{for all $t\in [0, T_{+})$}.
\]
Therefore, by using the scattering criterion, Theorem \ref{Th1}, we see that $u(t)$ scatters in $H^{1}(\R)$
forward in the time.\\
\textit{Case 2.} Suppose
\[
\sup_{t\in [0, \infty)}K(u(t))=K(Q).
\]
Then, by definition, there exists a sequence of positive times $t_{n}$ such that
\[
E(u(t_{n}))=E(Q),\quad M(u(t_{n}))=M(Q), \quad \lim_{n\rightarrow \infty}K(u(t_{n}))=K(Q).
\]
Notice that $t_{n}\rightarrow \infty$ as $n\rightarrow \infty$. Indeed, if $t_{n}\to t^{\ast}<\infty$, then, by continuity of the flow we have that  $u(t_{n})\rightarrow u(t^{\ast})$ in $H^{1}(\R)$, which implies that $K(u(t^{\ast}))=K(Q)$,  $M(u(t^{\ast}))=M(Q)$ and
$N(u(t^{\ast}))=N(Q)$.  By the same argument as above, there exists $\theta\in \R$ such that $u(x,t)=e^{i\theta}e^{i(t-t^{\ast})}Q(x)$, which is impossible because \eqref{ClaimKinecs}. This proves that $t^{\ast}=\infty$. Finally, we may use Lemma \ref{Compact} to obtain, up subsequence,
\[
u(t_{n}, \cdot)\rightarrow e^{i\theta} Q(\cdot) \quad \text{in $H^{1}(\R)$ as $n\rightarrow\infty$,}
\] 
for some $\theta\in \R$.\\

(ii) If  $u_{0}$ satisfies \eqref{IgualKinectMass}, then by the scaling \eqref{Scaling} we may suppose
\[
M(u_{0})=M(Q), \quad E(u_{0})=E(Q), \quad K(u_{0})=K(Q).
\]
In particular, it follows that $N(u_{0})=N(Q)$. From \eqref{Unique} we obtain that $N(u_{0})=M^{\frac{p+1}{4}}(u_{0})K^{\frac{p+1}{4}}(u_{0})$. Therefore, by uniqueness of minimizer we infer that $u_{0}(x)=e^{i\theta}Q(x)$ for some $\theta\in \R$. Finally, by using the uniqueness of the solution for \eqref{NLS} we have
$u(x,t)=e^{it}e^{i\theta}Q(x)$. This proves the statement (ii).\\

(iii) If $u_{0}$ satisfies \eqref{MayorlKinectMass}, again by the scaling \eqref{Scaling} we get

\begin{equation}\label{KQKu}
K(u_{0})>K(Q).
\end{equation}
By the same argument as in proof of Claim \eqref{ClaimKinecs} we can show that
\begin{equation}\label{KmayorQ}
K(u(t))>K(Q)\quad \text{for all $t\in [0, T_{+})$}.
\end{equation}
Next we consider two cases:\\
\textit{Case 1.} Suppose
\[
\sup_{t\in [0, T_{+})}K(u(t))>K(Q).
\]
Then there exists $\delta>0$ such that 
\[
K(u(t))\geq (1+\delta)K(Q) \quad \text{for all $t\in [0, T_{+})$}.
\]
As a consequence,  inequality above and \eqref{PohoIdenti} yield
\begin{align*}
	{N(u(t))}&=(p+1)\(\frac{1}{2}K(u(t))-E(u_{0})\)\\
	&\geq (p+1)\( \frac{(1+\delta)}{2}K(Q)-E(Q)\)\\
	&={N(Q)}+\frac{\delta(p+1)}{2}K(Q)\\
	&=\(1+\frac{\delta(p+1)}{4}\)N(Q)\quad \text{for all $t\in [0, T_{+})$}.
\end{align*}
Consequently, by inequality above and \eqref{EnergyIdentities} we see that
\begin{align*}
	G(M(t))M(u(t))&=16E(u_{0})M(u_{0})-\frac{4(p-3)}{(p+1)}N(u(t))M(u_{0})      \\
	&\leq 16E(Q)M(Q)-\(1+\frac{\delta(p+1)}{4}\)\(\frac{4(p-3)}{(p+1)}  \)N(Q)M(Q)\\
	&=-\frac{\delta(p+1)}{4}N(Q)M(Q), \quad \text{for all $t\in [0, T_{+})$}.
\end{align*}
By using Lemma \ref{Blow-up-Lemma} we then obtain that $T_{+}<\infty$ and therefore the solution $u(t)$ blows-up in finite time.\\
 \textit{Case 2.} Suppose
\[
\sup_{t\in [0, T_{+})}K(u(t))=K(Q).
\]
Then there exists a sequence $t_{n}\in [0, T_{+})$ such that $K(u(t))\rightarrow K(Q)$ as $n\to \infty$.
We may assume that, possibly for a subsequence only,  $t_{n}\rightarrow t^{\ast}\in [0, T_{+}]$. Now, 
suppose that $t^{\ast}<T_{+}<\infty$, then, by the same argument as above we infer that 
$u_{0}(x)=e^{i\theta}Q(x)$ for some $\theta\in \R$, which is a contradiction with \eqref{KQKu}. On the other hand, if 
$t^{\ast}=T_{+}<\infty$, then
\[
E(u(t_{n}))=E(Q),\quad M(u(t_{n}))=M(Q), \quad \lim_{n\rightarrow \infty}K(u(t_{n}))=K(Q).
\]
Thus, by Lemma \ref{Compact} there exists $\theta\in \R$ such that, up to a subsequence,
$u(t_{n})\rightarrow e^{i\theta} Q$ strongly in $H^{1}(\R)$ as $n\to \infty$. However,
by the blow-up alternative, this is impossible because $T_{+}<\infty$. In conclusion we have that
either $t^{\ast}<T_{+}=\infty$ or $t^{\ast}=T_{+}=\infty$. Again, if $t^{\ast}<T_{+}=\infty$, then $u_{0}(x)=e^{i\theta}Q(x)$,
which is impossible. Therefore, $t^{\ast}=T_{+}=\infty$. Hence \eqref{CoverGroundnew} follows from Lemma \ref{Compact}.
The proof of the theorem is complete.
\end{proof}

\section{Blow-up and applications}\label{Blow-up-Corollary}

In this section we show the Theorem \ref{BlowupboveQ} and  Corollaries \ref{GroundPhase} and \ref{LongGeral}.

\begin{proof}[{Proof of Theorem \ref{BlowupboveQ}}]
First we recall the definition of $z(t)$, $\Phi(x)$ and $x_{0}$ in \eqref{InVirial}, \eqref{Phifun} and \eqref{PosiIdent}, respectively.
By hypothesis \eqref{Blowup33}, we have
\begin{equation}\label{ZZeronega}
z^{\prime}(0)=\frac{V^{\prime}(0)}{2\sqrt{V(0)}}\leq 0.
\end{equation}
Moreover, the assumption \eqref{Blowup22} is equivalent to 
\[
\frac{(p-3)}{4(p+1)}\(\frac{M(u_{0})}{M(Q)}\)^{\sigma_{c}}\( \frac{N(u_{0})}{E(Q)}\)>1,
\]
where we have used \eqref{EnergyIdentities}. Thus, combining inequality above with identity \eqref{PosiIdent} and \eqref{Viria33} implies 
\begin{equation}\label{VNegative}
V^{\prime\prime}(0)<x_{0}.
\end{equation}
Now, by using the assumption \eqref{assump22} we get
\begin{equation}\label{Newboundz}
(z^{\prime}(0))^{2}\geq\frac{x_{0}}{2}=4\Phi(x_{0}).
\end{equation}
Therefore, using \eqref{VNegative} and \eqref{Newboundz} we obtain the inequality
\[
z^{\prime\prime}(0)=\frac{1}{z(0)}\(\frac{V^{\prime\prime}(0)}{2}-(z^{\prime}(0))^{2} \)
<\frac{1}{z(0)}\(\frac{x_{0}}{2}-\frac{x_{0}}{2}\)=0.
\]
Next, we show that
\begin{equation}\label{claimBlowup}
z^{\prime\prime}(t)<0, \quad \text{for all $t\in [0, T_{+})$}.
\end{equation}
Indeed, suppose by contradiction that for some $t^{\ast}\in [0, T_{+})$ we have $z^{\prime\prime}(t^{\ast})\geq 0$.
Since $z^{\prime\prime}(0)<0$, the intermediate value theorem implies that there exists $t_{0}\in (0, T_{+})$ such that
\[
z^{\prime\prime}(t_{0})=0 \quad \text{and} \quad z^{\prime\prime}(t)<0, \quad \text{for all $t\in [0, t_{0})$}.
\]
Therefore, as $z^{\prime}(0)\leq 0$, by \eqref{Newboundz} we obtain
\[
z^{\prime}(t)<z^{\prime}(0)\leq -\sqrt{4 \Phi(x_{0})},\quad \text{for all $t\in (0, t_{0}]$},
\]
which implies
\[
(z^{\prime}(t))^{2}> 4 \Phi(x_{0}) \quad \text{for all $t\in (0, t_{0}]$}.
\]
Then using \eqref{InVirial} we obtain
\[
4 \Phi(V^{\prime\prime}(t)) >4 \Phi(x_{0}), \quad \text{for all $t\in (0, t_{0}]$}.
\]
The last inequality combined with \eqref{VNegative} yields
\[
V^{\prime\prime}(t_{0})<x_{0} \quad \text{for all $t\in[0,t_{0}]$}.
\]
Finally, by using the inequality above and \eqref{Newboundz} we get
\[
z^{\prime\prime}(t_{0})=\frac{1}{z(t_{0})}\(\frac{V^{\prime\prime}(t_{0})}{2}-(z^{\prime}(t_{0}))^{2}\)
<\frac{1}{z(t_{0})}\(\frac{x_{0}}{2}- \frac{x_{0}}{2}\)=0,
\]
which is impossible by definition of $t_{0}$. Thus, \eqref{claimBlowup} holds.
Now, we proceed by contradiction. Suppose that  $T_{+}=\infty$. By $z^{\prime}(0)\leq 0$ and \eqref{claimBlowup}
we obtain
\[
z(t)=z(1)+\int^{t}_{1}z^{\prime}(t)dt<z(1)+z^{\prime}(1)(t-1)<0
\]
for $t$ large,
which is impossible because $z(t)$ is nonnegative. This completes the proof of theorem.
\end{proof}

\begin{proof}[{Proof of Corollary \ref{GroundPhase}}]Let $\gamma>0$. First we show that there exists $t_{0}$ such that $\psi^{\gamma}(t_{0})$ satisfies the assumptions \eqref{assum11}-\eqref{assump44} 
in Theorem \ref{ScatteaboveQ}. We recall that $\psi^{\gamma}(0)=e^{i\gamma x^{2}}Q(x)$ with $Q(x)=2^{\frac{1}{p-1}}e^{-|x|}$.
Notice that a direct calculation shows that
\begin{equation}\label{DerivateZ}
\partial_{x}\psi^{\gamma}(0)=e^{i\gamma x^{2}}(2i\gamma xQ+\partial_{x}Q),
\end{equation}
and therefore
\begin{equation}\label{ViReal}
\IM\int_{\R}x\partial_{x}\psi^{\gamma}(0)\overline{\psi^{\gamma}(0)}dx=2\gamma\int_{\R}x^{2}Q^{2}(x)dx>0.
\end{equation}
Then, by continuity we get 
\[
\IM\int_{\R}x\partial_{x}\psi^{\gamma}(t_{0})\overline{\psi^{\gamma}(t_{0})}>0
\]
for $t_{0}$ sufficiently small, i.e., the assumption \eqref{assump44} holds when $t_{0}$ is sufficiently small. Moreover, since
\begin{equation}\label{EnerIgual}
E(\psi^{\gamma}(t_{0}))=E(Q)+2\gamma\IM\int_{\R}x\partial_{x}\psi^{\gamma}(t_{0})\overline{\psi^{\gamma}(t_{0})}
+2\gamma^{2}\int_{\R}x^{2}|\psi^{\gamma}(t_{0})|^{2}dx,
\end{equation}
it follows that 
\[
E(\psi^{\gamma}(t_{0}))[M(\psi^{\gamma}(t_{0}))]^{\sigma_{c}}\geq E(Q)[M(Q)]^{\sigma_{c}},
\]
which implies that the assumption \eqref{assum11} holds. On the other hand, 
from equation \eqref{NLS}
we obtain
\begin{align*}
\partial_{t}N(\psi^{\gamma}(t))&=(p+1)|\psi^{\gamma}(0,t)|^{p-1}\RE[\overline{\psi^{\gamma}(0,t)}\partial_{t}\psi^{\gamma}0,t)]	\\
	&=(p+1)|\psi^{\gamma}(0,t)|^{p-1}\RE[\overline{\psi^{\gamma}(0,t)}(-i\partial^{2}_{x}\psi^{\gamma}(0,t)+ i |\psi^{\gamma}(0,t)|^{p-1}\psi^{\gamma}(0,t) )]\\
	&=-(p+1)|\psi^{\gamma}(0,t)|^{p-1}\IM[\overline{\psi^{\gamma}(0,t)}\partial^{2}_{x}\psi^{\gamma}(0,t)].
\end{align*}
By using the fact
\[
\partial^{2}_{x}\psi^{\gamma}(x,0)=e^{i\gamma x^{2}}(2i\gamma Q-4i\gamma 2^{\frac{1}{p-1}}|x|e^{-|x|}-4\gamma^{2}x^{2}Q+\partial^{2}_{x}Q)
\]
we have that
\begin{equation}\label{DerivateN1}
\left.\partial_{t}N(\psi^{\gamma}(t))\right|_{t=0}=-2\gamma(p+1)N(Q)<0.
\end{equation}
Thus, by using the fact
\[
[M(\psi^{\gamma}(0))]^{\sigma_{c}}N(\psi^{\gamma}(0))=M^{\sigma_{c}}(Q)N(Q)
\]
we conclude that assumption \eqref{assum33} holds for $t_{0}$ small. Now we set
\[
H(t)=[M(\psi^{\gamma})]^{\sigma_{c}}\(
E(\psi^{\gamma})-\frac{(\IM\int_{\R}x\partial_{x}\psi^{\gamma}(t)\overline{\psi^{\gamma}(t)})^{2}}{2\int_{\R}x^{2}|\psi^{\gamma}(t)|^{2}dx}\)
-[M(Q)]^{\sigma_{c}}E(Q)
\]
or, equivalently,
\begin{equation}\label{EquivaH}
H(t)=[M(\psi^{\gamma})]^{\sigma_{c}}\(
E(\psi^{\gamma})-\frac{1}{8}(z^{\prime}(t))^{2}\)
-[M(Q)]^{\sigma_{c}}E(Q)
\end{equation}
where (see proof of Theorem \ref{ScatteaboveQ})
\[
z(t)=\sqrt{V(t)}, \quad V(t)=\int_{\R}x^{2}|\psi^{\gamma}(x,t)|^{2}dx.
\]
Notice that putting together  \eqref{ViReal} and \eqref{EnerIgual} we deduce
\[
E(\psi^{\gamma})-\frac{(\IM\int_{\R}x\partial_{x}\psi^{\gamma}(0)\overline{\psi^{\gamma}(0)})^{2}}{2\int_{\R}x^{2}|\psi^{\gamma}(0)|^{2}dx}
=E(Q),
\]
which implies that $H(0)=0$. From \eqref{EquivaH}, we see that
\begin{equation}\label{Hderivate}
H^{\prime}(t)=-\frac{1}{4}[M(\psi^{\gamma})]^{\sigma_{c}}z^{\prime}(t)z^{\prime\prime}(t).
\end{equation}
Now, as a consequence of \eqref{ViReal} we get
\begin{equation}\label{VdigualV}
V^{\prime}(0)=8\gamma V(0).
\end{equation}
Moreover, since
\[
K(e^{i\gamma x^{2}}Q)=4\gamma^{2}V(0)+K(Q),
\]
it follows from \eqref{PohoIdenti} and \eqref{Viria11},
\begin{equation}\label{Vdosderivate}
\begin{split}
V^{\prime\prime}(0)&=8K(e^{i\gamma x^{2}}Q)-4N(e^{i\gamma x^{2}}Q)\\
&=32\gamma^{2}V(0)+8K(Q)-4N(Q)\\
&=32\gamma^{2}V(0).
\end{split}
\end{equation}
Then combining \eqref{VdigualV} and \eqref{Vdosderivate} we infer that
\[
(z^{\prime}(0))^{2}=\frac{1}{2}V^{\prime\prime}(0),
\]
which implies
\[
z^{\prime\prime}(0)=\frac{1}{z(0)}\(\frac{V^{\prime\prime}(0)}{2}-(z^{\prime}(0))^{2} \)
=0.
\]
As a consequence $H^{\prime}(0)=0$,
\[
H^{\prime\prime}(0)=-\frac{1}{4}[M(\psi^{\gamma})]^{\sigma_{c}}z^{\prime}(0)z^{\prime\prime\prime}(0).
\]
and
\[
V^{\prime\prime\prime}(0)=2z(0)z^{\prime\prime\prime}(0).
\]
Therefore, $H^{\prime\prime}(0)=-\frac{1}{8}[M(\psi^{\gamma})]^{\sigma_{c}}V^{\prime\prime\prime}(0)$. 
Finally, putting together \eqref{Viria33} and \eqref{DerivateN1} we have
\[
V^{\prime\prime\prime}(0)=-\frac{4(p-3)}{(p+1)}\left.\partial_{t}N(\psi^{\gamma}(t))\right|_{t=0}>0,
\]
and thus, $H^{\prime\prime}(0)<0$. In particular, this implies that $H(t)$ is negative when $t>0$ is small. Therefore, 
the assumption \eqref{assump22} holds for $t$ sufficiently small. Hence an application of Theorem \ref{ScatteaboveQ} shows that 
the solution $\psi^{\gamma}(t)$ scatters in $H^{1}(\R)$ forward in time.

Next, by the same argument as above  we may show that the solution $\overline{\psi^{\gamma}}(-t)$ to \eqref{NLS}
satisfies the assumptions \eqref{assump22},  \eqref{Blowup22} and \eqref{Blowup33} in Theorem \ref{BlowupboveQ}. Hence by Theorem \ref{ScatteaboveQ} 
we infer that the solution $\overline{\psi^{\gamma}}(-t)$ blow-up in positive time, i.e.,  $\psi^{\gamma}(t)$ blow up in negative time.
This proves the corollary when $\gamma>0$. The second part of the corollary, when $\gamma<0$, can be proved in a similar way.
\end{proof}

\begin{proof}[{Proof of Corollary \ref{LongGeral}}]
Consider the solution $u_{\mu}(t)$ of \eqref{NLS} with initial data $u_{\mu, 0}=e^{i\mu x^{2}}u_{0}$.
We will assume that  $\mu>0$ and $[M(u_{0})]^{\sigma_{c}}N(u_{0})<[M(Q)]^{\sigma_{c}}N(Q)$; the proof of statement (ii) is similar.

We consider two cases:\\
\textit{Case 1.} We first assume that
\begin{equation}\label{Step1}
	E(u_{\mu, 0})[M(u_{\mu, 0})]^{\sigma_{c}}\geq E(Q)[M(Q)]^{\sigma_{c}}.
\end{equation}
Then we will show that the initial data $u_{\mu, 0}$ satisfies  the hypotheses \eqref{assump22}-\eqref{assump44}  
in  Theorem \ref{ScatteaboveQ}. Indeed, a direct calculation shows that
\[\begin{split}
E(u_{\mu, 0})=E(u_{0})+2\mu\IM\int_{\R}x\cdot \partial_{x} u_{0}\overline{u_{0}}dx
+2\mu^{2}\int_{\R}|x|^{2}|u_{0}|^{2}dx
\end{split}\]
and
\begin{equation}\label{Viriequiva}
\begin{split}
\IM\int_{\R}x\cdot \partial_{x} u_{\mu,0}\overline{u_{\mu, 0}}dx=
\IM\int_{\R}x\cdot \partial_{x} u_{0}\overline{u_{0}}dx
+2\mu\int_{\R}|x|^{2}|u_{0}|^{2}dx.
\end{split}
\end{equation}
Combining the equations above, we obtain that
\[\begin{split}
E(u_{\mu, 0})-\frac{\( \IM\int_{\R}x\cdot \partial_{x} u_{\mu,0}\overline{u_{\mu, 0}}dx\)^{2}}{2\int_{\R}|x|^{2}|u_{\mu, 0}|^{2}dx}
=E(u_{0})-\frac{\( \IM\int_{\R}x\cdot \partial_{x} u_{0}\overline{u_{0}}dx\)^{2}}{2\int_{\R}|x|^{2}|u_{0}|^{2}dx}
\leq E(u_{0}),
\end{split}\]
or, equivalently
\begin{equation}\label{Conditionone}
\begin{split}
\frac{E(u_{\mu,0})[M(u_{\mu, 0})]^{\sigma_{c}}}{E(Q)[M(Q)]^{\sigma_{c}}}
\(1-\frac{\( \IM\int_{\R}x\cdot \partial_{x}  u_{\mu,0}\overline{u_{\mu, 0}}dx\)^{2}}{2E(u_{0})\int_{\R}|x|^{2}|u_{0}|^{2}dx}\)
\\
=\frac{E(u_{0})[M(u_{0})]^{\sigma_{c}}}{E(Q)[M(Q)]^{\sigma_{c}}}\leq 1,
\end{split}
\end{equation}
where we have used  \eqref{LongCondi11}. Notice that \eqref{Conditionone} implies the assumption \eqref{assump22}.
Moreover, as $[M(u_{\mu,0})]^{\sigma_{c}}N(u_{\mu, 0})=[M(u_{0})]^{\sigma_{c}}N(u_{0})$, it is clear that the assumption \eqref{assum33}
 of Theorem \ref{ScatteaboveQ} is fulfilled. 

Next we consider the  quadratic polynomial in $r$,
\[\begin{split}
P(r):=E(u_{0})+2r\IM\int_{\R}x\cdot \partial_{x}u_{0}\overline{u_{0}}dx+
2r^{2}\int_{\R}|x|^{2}|u_{0}|^{2}dx-\frac{E(Q)[M(Q)]^{\sigma_{c}}}{[M(u_{0})]^{\sigma_{c}}}.
\end{split}\]
Notice that
\[
[M(u_{\mu, 0})]^{\sigma_{c}}P(\mu)=[M(u_{\mu, 0})]^{\sigma_{c}}E(u_{\mu, 0})\(1-\frac{E(Q)[M(Q)]^{\sigma_{c}}}{E(u_{\mu, 0})[M(u_{0})]^{\sigma_{c}}}
\).
\]
Thus, using \eqref{Step1} we obtain that $P(\mu)\geq 0$. Moreover, assumption \eqref{LongCondi11} is equivalent to  $P(0)<1$.
Therefore, there exists $r_{0}\geq 0$ such that $\mu\geq r_{0}$, where $r_{0}$ satisfies $P(r_{0})=0$.  Since $P(r_{0})=0$,
by \eqref{LongCondi11} we obtain the inequality
\[
\IM\int_{\R}x\cdot \partial_{x} u_{0}\overline{u_{0}}dx+
r_{0}\int_{\R}|x|^{2}|u_{0}|^{2}dx\geq 0.
\]
Thus, by \eqref{Viriequiva} we conclude
\[
\IM\int_{\R}x\cdot \partial_{x}  u_{\mu, 0}\overline{u_{\mu, 0}}dx
\geq \mu\int_{\R}|x|^{2}|u_{0}|^{2}dx,
\]
which yields \eqref{assump44}. In view of Theorem \ref{ScatteaboveQ}, this implies that solution scatters forward in time.\\
\textit{Case 2.} Now we suppose that

\begin{equation}\label{Step22}
	E(u_{\mu, 0})[M(u_{\mu, 0})]^{\sigma_{c}}< E(Q)[M(Q)]^{\sigma_{c}}.
\end{equation}
We show that condition $[M(u_{0})]^{\sigma_{c}}N(u_{0})<[M(Q)]^{\sigma_{c}}N(Q)$ implies \eqref{KinectMass}.
Indeed, from \eqref{Step22} we have that
\[\begin{split}
 E(Q)[M(Q)]^{\sigma_{c}}>E(u_{\mu, 0})[M(u_{\mu, 0})]^{\sigma_{c}}
>\frac{1}{2}K(u_{\mu, 0})[M(u_{\mu, 0})]^{\sigma_{c}}-\frac{1}{p+1}K(Q)[M(Q)]^{\sigma_{c}},
\end{split}\]
combining the first and last term, we infer that $u_{\mu, 0}$ satisifies \eqref{KinectMass}. Hence, by \eqref{Step22} and an application of Theorem \ref{ScattebelowQ} completes the proof.
\end{proof}

\section*{Acknowledgments} 
The author would like to express their sincere thanks to the referees for many helpful comments.

\bibliographystyle{siam}
\bibliography{bibliografia}

\begin{thebibliography}{10}

\bibitem{AdamiFukuHolmer}
{\sc R.~Adami, R.~Fukuizumi, and J.~Holmer}, {\em Scattering for the ${L}^2$
  supercritical point {NLS}}, Trans. Amer. Math. Soc., 374 (2021), pp.~35--60.

\bibitem{ADAMI2001148}
{\sc R.~Adami and A.~Teta}, {\em A class of nonlinear {S}chr\"odinger equations
  with concentrated nonlinearity}, J. of Funct. Anal., 180 (2001), pp.~148
  --175.

\bibitem{BMQT}
{\sc B.~Bellazi and M.~Mintchev}, {\em Quantum field theory on star graphs},
  Phys. A: Math. Theor., 39 (2006), pp.~1101--1117.

\bibitem{QGAH}
{\sc G.~Berkolaiko, C.~Carlson, S.~Fulling, and P.~Kuchment}, {\em Quantum
  Graphs and Their Applications}, vol.~415 of Contemporary Math., American
  Math. Society, Providence, RI, 2006.

\bibitem{AFI}
{\sc V.~Caudrelier, M.~Mintchev, and E.~Ragoucy}, {\em Solving the quantum
  nonlinear {\Large s}chrödinger equation with $\delta$-type impurity}, J.
  Math. Phys., 4 (2005), pp.~1--24.

\bibitem{DinhForceHaja2020}
{\sc V.~Dinh, L.~Forcella, and H.~Hajaiej}, {\em Mass-{E}nergy threshold
  dynamics for dipolar {Q}uantum {G}ases}, Preprint arXiv:2009.05933,  (2020),
  p.~31 pages.

\bibitem{Dinh2020}
{\sc V.~D. Dinh}, {\em A unified approach for energy scattering for focusing
  nonlinear {S}chr\"odinger equations}, Discrete Contin. Dyn. Syst., 40 (2020),
  pp.~6441--6471.

\bibitem{DodsonMurphy2018}
{\sc B.~Dodson and J.~Murphy}, {\em A new proof of scattering below the ground
  state for the non-radial focusing {NLS}}, Math. Res. Lett., 25 (2018),
  pp.~1805 --1825.

\bibitem{DuyHolmerRoude2008}
{\sc T.~Duyckaerts, J.~Holmer, and S.~Roudenko}, {\em Scattering for the
  non-radial 3{D} cubic nonlinear {S}chr\"odinger equation}, Math. Res. Lett.,
  15 (2008), pp.~1233--1250.

\bibitem{DR5}
{\sc T.~Duyckaerts and S.~Roudenko}, {\em Going beyond the threshold:
  scattering and blow-up in the focusing {NLS} equation}, Comm. Math. Phys, 334
  (2015), pp.~1573--1615.

\bibitem{RJJ}
{\sc R.~Fukuizumi and L.~Jeanjean}, {\em Stability of standing waves for a
  nonlinear {\Large s}chrödinger equation with a repulsive {D}irac delta
  potential}, Discrete Contin. Dyn. Syst., 21 (2008), pp.~121--136.

\bibitem{FO}
{\sc R.~Fukuizumi, M.~Ohta, and T.~Ozawa}, {\em Nonlinear {\Large s}chrödinger
  equation with a point defect}, Ann. Inst. H. Poincaré Anal. Non Linéaire, 25
  (2008), pp.~837--845.

\bibitem{GaoWang2020}
{\sc Y.~Gao and Z.~Wang}, {\em Below and beyond the mass--energy threshold:
  scattering for the {H}artree equation with radial data in $d\geq 5$}, Angew.
  Math. Phys., 71 (2020).

\bibitem{GHW}
{\sc H.~Goodman, P.~Holmes, and M.~Weinstein}, {\em Strong {NLS}
  soliton--defect interactions}, Physica D, 192 (2004), pp.~215--248.

\bibitem{Guevara2012}
{\sc C.~D. Guevara}, {\em Global behavior of finite energy solutions to the
  d-dimensional focusing nonlinear {S}chr\"odinger equation}, Appl. Math. Res.
  Express, 2014 (2013), pp.~177--243.

\bibitem{HOLMER2020123522}
{\sc J.~Holmer and C.~Liu}, {\em Blow-up for the 1{D} nonlinear {S}chr\"odinger
  equation with point nonlinearity {I}: {B}asic theory}, J. Math. Anal. Appl.,
  483 (2020), p.~123522.

\bibitem{HolmerRoudenko2008}
{\sc J.~Holmer and S.~Roudenko.}, {\em A sharp condition for scattering of the
  radial 3{D} cubic nonlinear {S}chr\"odinger equation}, Comm. Math. Phys., 282
  (2008), pp.~435--467.

\bibitem{IkeaInu2017}
{\sc M.~Ikeda and T.~Inui}, {\em Global dynamics below the standing waves for
  the focusing semilinear {S}chr\"odinger equation with a repulsive {D}irac
  delta potential}, Anal. PDE, 10 (2017), pp.~481--512.

\bibitem{KenigMerle2006}
{\sc C.~E. Kenig and F.~Merle}, {\em Global well-posedness, scattering and
  blow-up for the energy-critical, focusing, non-linear {S}chr\"odinger
  equation in the radial case}, Invent. Math, 166 (2006), pp.~645--675.

\bibitem{KiiVisan2008}
{\sc R.~Killip and M.~Visan}, {\em Nonlinear {S}chr\"odinger equations at
  critical regularity}, in Lecture notes of the 2008 Clay summer school
  "Evolution Equations", 2008.

\bibitem{LFF}
{\sc S.~{Le~Coz}, R.~Fukuizumi, G.~Fibich, B.~Ksherim, and Y.~Sivan}, {\em
  Instability of bound states of a nonlinear {\Large s}chrödinger equation with
  a dirac potential}, Phys. D, 237 (2008), pp.~1103--1128.

\bibitem{MakMaloCho2003}
{\sc W.~C.~K. Mak, B.~A. Malomed, and P.~L. Chu}, {\em Interaction of a soliton
  with a local defect in a fiber {B}ragg grating}, J. Opt. Soc. Am. B, 20
  (2003), pp.~725--735.

\bibitem{PDEOM}
{\sc F.~A. Mehmeti, J.~von Below, and S.~Nicaise}, eds., {\em Partial
  Differential equations on multistrucutres}, no.~219 in Lecture Notes in Pure
  and Applied Mathematics, Marcel Dekker, Inc., New York, 2001.

\bibitem{Tarek2020}
{\sc T.~Saanouni}, {\em Scattering versus blow-up beyond the threshold for the
  focusing choquard equation}, J. Math. Anal. Appl., 492 (2020).

\end{thebibliography}

\end{document}